\documentclass{amsart}

\usepackage{amssymb, amsmath, amsthm, amscd, mathrsfs}

\usepackage{verbatim}
\usepackage{a4wide}

\usepackage{enumerate}

\allowdisplaybreaks

\theoremstyle{plain}
\newtheorem{theorem}{Theorem}[section]
\theoremstyle{remark}
\newtheorem{remark}[theorem]{Remark}

\theoremstyle{plain}
\newtheorem{corollary}[theorem]{Corollary}
\newtheorem{lemma}[theorem]{Lemma}
\newtheorem{proposition}[theorem]{Proposition}

\numberwithin{equation}{section}

\def\N{{\mathbb N}}

\def\R{{\mathbb R}}

\def\C{{\mathbb C}}

\renewcommand{\P}{{\mathbb P}}
\newcommand{\F}{{\mathscr F}}

\renewcommand{\O}{\Omega}

\newcommand{\calL}{{\mathscr L}}

\newcommand{\one}{{{\bf 1}}}

\newcommand{\lb}{\langle}
\newcommand{\rb}{\rangle}

\newcommand{\wh}{\widehat}
\newcommand{\supp}{\text{\rm supp\,}}
\newcommand{\nn}{|\!|\!|}

\newcommand{\Schw}{{\mathscr S}}
\newcommand{\TD}{{\mathscr S'}}

\newcommand{\Rad}{\rm{rad}}

\def\typeout#1{\message{^^J}\message{#1}\message{^^J}}
\newif\ifSRCOK \SRCOKtrue
\newcount\PAGETOP
\newcount\LASTLINE
\global\PAGETOP=1
\global\LASTLINE=-1
\def\EJECT{\SRC\eject}
\def\WinEdt#1{\typeout{:#1}}
\gdef\MainFile{\jobname.tex}
\gdef\CurrentInput{\MainFile}
\newcount\INPSP
\global\INPSP=0
\def\SRC{\ifSRCOK%
  \ifnum\inputlineno>\LASTLINE%
    \ifnum\LASTLINE<0%
      \global\PAGETOP=\inputlineno%
    \fi%
    \global\LASTLINE=\inputlineno%
    \ifnum\INPSP=0%
      \ifnum\inputlineno>\PAGETOP%
        
      \fi%
    \else%
      
    \fi%
  \fi%
\fi}
\def\PUSH#1{%
\SRC%
\ifnum\INPSP=0 \global\let\INPSTACKA=\CurrentInput \else%
\ifnum\INPSP=1 \global\let\INPSTACKB=\CurrentInput \else%
\ifnum\INPSP=2 \global\let\INPSTACKC=\CurrentInput \else%
\ifnum\INPSP=3 \global\let\INPSTACKD=\CurrentInput \else%
\ifnum\INPSP=4 \global\let\INPSTACKE=\CurrentInput \else%
\ifnum\INPSP=5 \global\let\INPSTACKF=\CurrentInput \else%
               \global\let\INPSTACKX=\CurrentInput \fi\fi\fi\fi\fi\fi%
\gdef\CurrentInput{#1}%
\WinEdt{<+ \CurrentInput}%
\global\LASTLINE=0%
\ifSRCOK\fi%
\global\advance\INPSP by 1}
\def\POP{%
\ifnum\INPSP>0 \global\advance\INPSP by -1  \fi%
\ifnum\INPSP=0 \global\let\CurrentInput=\INPSTACKA \else%
\ifnum\INPSP=1 \global\let\CurrentInput=\INPSTACKB \else%
\ifnum\INPSP=2 \global\let\CurrentInput=\INPSTACKC \else%
\ifnum\INPSP=3 \global\let\CurrentInput=\INPSTACKD \else%
\ifnum\INPSP=4 \global\let\CurrentInput=\INPSTACKE \else%
\ifnum\INPSP=5 \global\let\CurrentInput=\INPSTACKF \else%
               \global\let\CurrentInput=\INPSTACKX \fi\fi\fi\fi\fi\fi%
\WinEdt{<-}%
\global\LASTLINE=\inputlineno%
\global\advance\LASTLINE by -1%
\SRC}
\def\INPUT#1{\relax}

\let\originalxxxeverypar\everypar
\newtoks\everypar
\originalxxxeverypar{\the\everypar\expandafter\SRC}
\everymath\expandafter{\the\everymath\expandafter\SRC}
\output\expandafter{\expandafter\SRCOKfalse\the\output}

\newif\ifSRCOK \SRCOKtrue
\newcount\PAGETOP
\newcount\LASTLINE
\global\PAGETOP=1
\global\LASTLINE=-1
\gdef\MainFile{\jobname.tex}
\gdef\CurrentInput{\MainFile}
\newcount\INPSP
\global\INPSP=0
\def\EJECT{\SRC\eject}
\def\WinEdt#1{\typeout{:#1}}
\def\SRC{\ifSRCOK%
  \ifnum\inputlineno>\LASTLINE%
    \ifnum\LASTLINE<0%
      \global\PAGETOP=\inputlineno%
    \fi%
    \global\LASTLINE=\inputlineno%
    \ifnum\INPSP=0%
      \ifnum\inputlineno>\PAGETOP%
      \fi%
    \else%
    \fi%
  \fi%
\fi}
\def\PUSH#1{%
\SRC%
\ifnum\INPSP=0 \global\let\INPSTACKA=\CurrentInput \else%
\ifnum\INPSP=1 \global\let\INPSTACKB=\CurrentInput \else%
\ifnum\INPSP=2 \global\let\INPSTACKC=\CurrentInput \else%
\ifnum\INPSP=3 \global\let\INPSTACKD=\CurrentInput \else%
\ifnum\INPSP=4 \global\let\INPSTACKE=\CurrentInput \else%
\ifnum\INPSP=5 \global\let\INPSTACKF=\CurrentInput \else%
               \global\let\INPSTACKX=\CurrentInput \fi\fi\fi\fi\fi\fi%
\gdef\CurrentInput{#1}%
\WinEdt{<+ \CurrentInput}%
\global\LASTLINE=0%
\ifSRCOK\fi%
\global\advance\INPSP by 1}
\def\POP{%
\ifnum\INPSP>0 \global\advance\INPSP by -1  \fi%
\ifnum\INPSP=0 \global\let\CurrentInput=\INPSTACKA \else%
\ifnum\INPSP=1 \global\let\CurrentInput=\INPSTACKB \else%
\ifnum\INPSP=2 \global\let\CurrentInput=\INPSTACKC \else%
\ifnum\INPSP=3 \global\let\CurrentInput=\INPSTACKD \else%
\ifnum\INPSP=4 \global\let\CurrentInput=\INPSTACKE \else%
\ifnum\INPSP=5 \global\let\CurrentInput=\INPSTACKF \else%
               \global\let\CurrentInput=\INPSTACKX \fi\fi\fi\fi\fi\fi%
\WinEdt{<-}%
\global\LASTLINE=\inputlineno%
\global\advance\LASTLINE by -1%
\SRC}
\def\INPUT#1{\relax}
\let\OldINCLUDE=\include
\def\include#1{
\EJECT%
\PUSH{#1.tex}%
\OldINCLUDE{#1}%
\POP}

\let\originalxxxeverypar\everypar
\newtoks\everypar
\originalxxxeverypar{\the\everypar\expandafter\SRC}
\everymath\expandafter{\the\everymath\expandafter\SRC}
\let\zzzxxxbibliography=\bibliography
\def\bibliography#1{\PUSH{\jobname.bbl}\zzzxxxbibliography{#1}\POP}
\output\expandafter{\expandafter\SRCOKfalse\the\output}

\begin{document}

\author{Martin Meyries}
\address{Institute of Mathematics,
Martin-Luther-Universit\"at Halle-Wittenberg, 06099 Halle (Saale), Germany}
\email{martin.meyries@mathematik.uni-halle.de}

\author{Mark Veraar}
\address{Delft Institute of Applied Mathematics\\
Delft University of Technology \\ P.O. Box 5031\\ 2600 GA Delft\\The
Netherlands} \email{M.C.Veraar@tudelft.nl}




\thanks{The first author was partially supported by the Deutsche Forschungsgemeinschaft (DFG)}

\title[Pointwise multiplication on vector-valued function spaces]{Pointwise multiplication on vector-valued function spaces with power weights}

\maketitle
\begin{abstract}
We investigate pointwise multipliers on vector-valued function spaces over $\R^d$, equipped with Muckenhoupt weights. The main result is that in the natural parameter range, the characteristic function of the half-space is a pointwise multiplier on Bessel-potential spaces with values in a UMD Banach space. This is proved for a class of power weights, including the unweighted case, and extends the classical result of Shamir and Strichartz. The multiplication estimate is based on the paraproduct technique and a randomized Littlewood-Paley decomposition. An analogous result is obtained for Besov and Triebel-Lizorkin spaces.
\end{abstract}


\section{Introduction}
It is a classical result of Shamir \cite{Shamir} and Strichartz \cite{Strich67} that for $p\in (1,\infty)$ the characteristic function $\one_{\R_+^d}$ of the half-space $\R^d_+ = \{(x',t): x'\in \R^{d-1},\, t>0\}$ acts as a pointwise multiplier on the Bessel-potential space (or fractional Sobolev space) $H^{s,p}(\R^d)$ in the parameter range
$$-\frac{1}{p'}<s<\frac1p,$$
where $p'$ is the dual exponent of $p$. This condition can be understood by recalling that the trace at the hyperplane $\{(x',0):x'\in \R^{d-1}\}$ is continuous on these spaces if and only if $s>1/p$. The case of negative smoothness follows from a duality argument. The corresponding result was proved some years earlier for the Slobodetskii spaces $W^{s,p}(\R^d)$ by Lions $\&$ Magenes \cite{LiMa} and Grisvard \cite{Gri63}. Further extensions to Besov spaces $B_{p,q}^s(\R^d)$ and Triebel-Lizorkin spaces $F_{p,q}^s(\R^d)$ were given by Peetre \cite{Peetre76}, Triebel \cite{Tri83}, Franke \cite{Franke86}, Marschall \cite{Marschall} and Sickel \cite{Sickel87}, see the monograph of Runst $\&$ Sickel \cite{RS96} for details. For more recent results we also refer to Sickel \cite{Sickel99b, Sickel99a} and Triebel \cite{Triebel02}.\smallskip

The characteristic function serves as a natural extension operator for the half-space. Its multi\-plier property was one of the main ingredients for Seeley's result \cite{Se} on complex interpolation of Bessel-potential spaces with  boundary conditions. On the other hand, the multiplier property is also a direct consequence of Seeley's result. In this sense the assertions are equivalent. They are further equivalent to the validity of Hardy's inequality \cite[Section 2.8.6]{Tri83}.\smallskip

In this paper we extend the multiplier result for the characteristic function  to the weighted vector-valued case. We consider power weights $w_\gamma$ depending on the last coordinate only, i.e.,
$$w_\gamma(x',t) = |t|^\gamma, \qquad x'\in \R^{d-1},\qquad t\in \R.$$
These weights act at the same hyperplane as $\one_{\R_+^d}$. Hence the parameter range where $\one_{\R_+^d}$ is a multiplier will depend on the exponent $\gamma$. Here the dual exponent $\gamma' = - \frac{\gamma}{p-1}$ of $\gamma$ with respect to $p$ comes into play.

The following is our main result. It is proved in Section \ref{subs:multich}.  In the vector-valued case it seems to be new also in the unweighted case $\gamma = 0$.

\begin{theorem} \label{thm:1}Let $X$ be a \emph{UMD} Banach space, $p\in (1,\infty)$ and $\gamma\in (-1,p-1)$. Then for
$$-\frac{1+\gamma'}{p'} < s < \frac{1+\gamma}{p}$$
the characteristic function $\one_{\R_+^d}$ of the half-space is a pointwise multiplier on $H^{s,p}(\R^d,w_\gamma;X)$.
\end{theorem}
To be precise, the theorem states that for all $f\in H^{s,p}(\R^d,w_\gamma;X)$ the product $\one_{\R_+^d}f$ again belongs to $H^{s,p}(\R^d,w_\gamma;X)$ and there is a constant $C > 0$, independent of $f$, such that
$$\|\one_{\R_+^d}f\|_{H^{s,p}(\R^d,w_\gamma;X)} \leq C \|f\|_{H^{s,p}(\R^d,w_\gamma;X)}.$$
This multiplier result seems to close a gap in the literature. It has already been used in several works.

The spaces $H^{s,p}(\R^d,w_\gamma;X)$ are defined with the Bessel-potential in the usual way based on the weighted Lebesgue space $L^p(\R^d,w_\gamma;X)$, see Section \ref{subsec:spaces}. For an exponent $\gamma\in (-1,p-1)$ as in the theorem, the weight $w_\gamma$ belongs to the Muckenhoupt class $A_p$, see Section \ref{sec:Mucken}. The condition on $s$ shows the effect of the weight on the regularity of $H^{s,p}(\R^d,w_\gamma;X)$ at the hyperplane $\{(x',0):x'\in \R^{d-1}\}$: the range of $s$ where jumps are allowed is enlarged as $\gamma$ increases.

A Banach space $X$ has UMD if and only if the Hilbert transform extends continuously to $L^2(\R;X)$, see Section \ref{subsec:UMD} for some details and references. For instance, Hilbert spaces and classical function spaces like $L^p$, $W^{s,p}$, $H^{s,p}$, $B_{p,q}^s$ and $F_{p,q}^s$ have UMD in their reflexive range. Many other fundamental operators in vector-valued harmonic analysis are bounded if and only if the underlying space has UMD. Since the 80's,
it has turned out that in UMD spaces one can develop vector-valued Fourier analysis (see \cite{Bou83, Bou2, Bu2, Bu3, McC84, Zim89}). More recently, this has led to an extensive theory on operator-valued Fourier multipliers and singular integrals (see \cite{GW03,HHN, HytTb, HytWeT1, StrWe,We}), which originally was motivated by regularity theory for parabolic PDEs (see \cite{DHP,KuWe} and references therein).

Employing standard localization techniques, Theorem \ref{thm:1} extends to the characteristic function of Lipschitz domains on spaces equipped with power weights based on the distance to the boundary.

\begin{corollary}\label{cor:Lipschitz-indicator}
Let $X$ be a \emph{UMD} Banach space, $p\in (1,\infty)$, $\gamma\in (-1,p-1)$ and $-\frac{1+\gamma'}{p'} < s < \frac{1+\gamma}{p}$. Assume $\Omega \subset \R^d$ is a bounded Lipschitz domain. Then the characteristic function $\one_{\Omega}$ of $\Omega$ is a pointwise multiplier on $H^{s,p}(\R^d,\emph{\text{dist}}(\cdot,\partial \Omega)^\gamma;X)$.
\end{corollary}

Our second main result concerns Besov and Triebel-Lizorkin spaces and does not require the UMD property of the underlying Banach space. It is also proved in Section \ref{subs:multich}. For weighted vector-valued $B$-spaces, the case $s> 0$ was already treated by Grisvard \cite{Gri63}.

\begin{theorem} \label{thm:2}Let $X$ be a Banach space, $p\in (1,\infty)$, $q\in [1,\infty]$ and $\gamma\in (-1,p-1)$. Then for
$-\frac{1+\gamma'}{p'} < s < \frac{1+\gamma}{p}$
the characteristic function $\one_{\R_+^d}$ of the half-space is a pointwise multiplier on $B^{s}_{p,q}(\R^d,w_\gamma;X)$ and on $F^{s}_{p,q}(\R^d,w_\gamma;X)$, respectively.
\end{theorem}
Corollary \ref{cor:Lipschitz-indicator} can be extended to the setting of $B$- and $F$-spaces as well.

Another result, which is also due to Strichartz \cite{Strich67} in the unweighted scalar case, is devoted to the pointwise multiplication with bounded $H^{s,p}$-functions and motivated by power nonlinearities.  A special case of Theorem \ref{thm:lastone} is the following, where we can allow for general weights $w\in A_p$. The notion of the type of a Banach space is explained in Section \ref{sec:type}. For instance, in the theorem one can choose for $X$ one of the classical function spaces $L^r$, $W^{\alpha,r}$, $H^{\alpha,r}$, $B_{r,q}^\alpha$ or $F_{r,q}^\alpha$, provided $q,r\in [2,\infty)$.

\begin{theorem}\label{thm:3} Let $X$ be a \emph{UMD} Banach space which has type $2$, and let $s>0$, $p\in(1,\infty)$ and $w\in A_p$. Then there is a constant $C> 0$ such that for all $m\in H^{s,p}(\R^d, w)\cap L^\infty(\R^d)$ and  $f\in H^{s,p}(\R^d,w;X)\cap L^\infty(\R^d;X)$ one has
\begin{align*}
\|mf\|_{H^{s,p}(\R^d,w;X)} \leq C \big(\|m\|_{L^\infty(\R^d)} \|f\|_{H^{s,p}(\R^d,w;X)} + \|m\|_{H^{s,p}(\R^d, w)} \|f\|_{L^\infty(\R^d;X)}\big).
\end{align*}
\end{theorem}

In Proposition  \ref{multiplication-Hinfty} a variant of this estimate is given for operator-valued multipliers $m$, i.e., $m(x)\in \calL(X,Y)$ for  $x\in \R^d$ with UMD spaces $X,Y$. In this case we have to assume that the image of $m$ is $\mathcal R$-bounded (see Section \ref{subsec:UMD} for more information), the sup-norm in the above estimate is replaced by its $\mathcal R$-bound $\mathcal R(m)$ and the $H$-norm of $m$ is replaced by an $F$-norm depending on the type of $Y$.\smallskip

Our motivation to consider the weighted vector-valued setting is the maximal $L^p$-$L^q$-maximal regularity approach to parabolic evolution equations, and further the approach based on the weights $\text{dist}(\cdot,\partial\Omega)^\gamma$ to treat problems with rough boundary data. In the forthcoming paper \cite{MeyVer4} we apply the multiplier results to extend Seeley's  characterization of complex and real interpolation spaces of Sobolev spaces with boundary conditions to the weighted vector-valued case. This allows, for instance, to characterize the fractional power domains of the time derivative with zero initial conditions on $L^p(\R_+,w_\gamma;X)$ and on $F_{p,q}^0(\R_+,w_\gamma;X)$.\smallskip

In the rest of this introduction we explain the techniques employed in the proofs of the above results and the difficulties arising in the vector-valued setting. \smallskip

Strichartz' proof of the multiplier assertion for the characteristic function on $H^{s,p}(\R^d)$ is based on a difference norm for these spaces, see \cite[Section 2]{Strich67}. It generalizes to $H$-spaces with values in a Hilbert space, see \cite[Section 6.1]{Walker}. In the general vector-valued case, such a norm does not seem to be available for $H^{s,p}(\R^d;X)$, even if $X$ has UMD. The vector-valued analogue of the difference norm leads to the Triebel-Lizorkin space $F^{s}_{p,2}(\R^d;X)$, see \cite[Section 2.5.10]{Tri83} and \cite[Theorem 6.9]{Triebel3}. But one has
$$ H^{s,p}(\R^d;X) = F^{s}_{p,2}(\R^d;X),$$
i.e., the usual Littlewood-Paley decomposition for the $H$-spaces, if and only if $X$ can be renormed as a Hilbert space (see \cite{HaMe96}, and Proposition \ref{prop:type-embed} for a refinement of this assertion in terms of type and cotype of $X$). As a substitute, a randomized Littlewood-Paley decomposition is available if $X$ has UMD. This result is originally due to Bourgain \cite{Bou2} and McConnell \cite{McC84}. In Section \ref{section:UMD} we derive such a decomposition for the weighted spaces $H^{s,p}(\R^d,w;X)$ with $A_p$-weights $w$, essentially as a consequence of \cite{HH12}. As a byproduct, we also obtain that these spaces form a complex interpolation scale.

In the general vector-valued case, difference norms are still available for $F$- and $B$-spaces with positive smoothness. As in \cite[Section 2.8.6]{Tri83} one could use these norms to prove the multiplier property of $\one_{\R_+^d}$. At least in the reflexive range, the case of negative smoothness then follows from a duality argument. However, this excludes the important cases $F_{p,1}^s$ and $F_{p,\infty}^s$ as well as non-reflexive underlying spaces $X$. For the Slobodetskii spaces $W$ and Besov spaces $B$, the multiplier result can also be derived as in \cite{Gri63} from real interpolation with Dirichlet boundary conditions. For real interpolation spaces quite convenient norms are available.  However, the $H$- and the $F$- spaces cannot be obtained by real interpolation.\smallskip

Theorem \ref{thm:1} will be a consequence of the following estimate, which is valid under the assumptions on the parameters as in the theorem:
\begin{equation}\label{intro-est}
 \|mf\|_{H^{s,p}(\R^d,w_\gamma;X)} \leq C \big( \|m\|_{B_{r,\infty}^{\frac{1+\mu}{r}}(\R,w_\mu)} + \|m\|_\infty\big)\|f\|_{H^{s,p}(\R^d,w_\gamma;X)}
\end{equation}
Here $r$ and $\mu$ can be chosen in a range which depends on the other parameters and $m$ only depends on the last coordinate of $\R^d$. After a suitable cut-off, the characteristic function $\one_{\R_+^d}$ belongs to the Besov space $B_{r,\infty}^{\frac{1+\mu}{r}}(\R,w_\mu)$ for all $r\in (1,\infty)$ and $\mu\in (-1,r-1)$ (see Lemma \ref{chi}). In this context $\one_{\R_+^d}$ is considered to depend on the last variable only. Together with \eqref{intro-est}, this yields Theorem \ref{thm:1}.

The estimate \eqref{intro-est} is shown in Theorem \ref{multiplication-esti}. Also here the more general case of an operator-valued $m$ is considered, where as before the sup-norm of $m$ is replaced by its $\mathcal R$-bound. It is analogous as for unweighted, scalar-valued $B$- and $F$-spaces, see \cite{Franke86, Marschall, RS96, Sickel87, Tri83} and in particular \cite[Section 4.6]{RS96}. Similar to these references, its proof is based on the paraproduct technique as introduced by Bony (see e.g. \cite{Bony}). For pointwise multipliers this method was first employed by Peetre \cite{Peetre76} and Triebel \cite{Tri83} in order to treat the case of $B$- and $F$-spaces in the full parameter range $p,q\in (0,\infty]$. For more recent developments in the context of paraproducts in a UMD-valued setting we refer to \cite{HytWeis10, MP12}.

The idea of the paraproduct approach is as follows, see also \cite[Section 4.4]{RS96}. For a function $\varphi$ with $\wh \varphi \in C_c^\infty(\R^d)$ and $\wh \varphi (0) = 1$ one sets $S^l f = \F^{-1}(\wh \varphi(2^{-l}\cdot) \wh f)$, such that $S^l f \to f$ as $l\to \infty$ in the sense of distributions. One defines the product of two distributions $m$ and $f$ as
$$mf = \lim_{l\to \infty} S^l m \cdot S^l f,$$
whenever this limit exists in the distributional sense. This extends the pointwise product of smooth functions. Observe that $S^l m \cdot S^l f$ is well-defined in a pointwise sense since the factors have compact Fourier support and are therefore smooth. Now one decomposes this limit into the sum of three series $\Pi_1(m,f)$, $\Pi_2(m,f)$ and $\Pi_3(m,f)$, the paraproducts, such that
$$mf = \Pi_1(m,f) + \Pi_2(m,f) + \Pi_3(m,f),$$
see Section \ref{sec:paraproducts} for details. These collect different sizes of Fourier supports of $m$ and $f$, respectively, and are thus estimated in different ways.

The estimate of $\Pi_1(m,f)$, in which the $m$-factors have large Fourier supports, is based on the randomized Littlewood-Paley decomposition for  $H^{s,p}(\R^d,w;X)$. It yields
\begin{equation}\label{intro-esti-1}
 \|\Pi_1(m,f)\|_{H^{s,p}(\R^d,w;X)} \leq C \|m\|_\infty \|f\|_{H^{s,p}(\R^d,w;X)},
\end{equation}
see Lemma \ref{multiplication1}. An analogous result holds with $H^{s,p}$ replaced by $F_{p,q}^s$ and $B_{p,q}^s$, where one can directly use the Littlewood-Paley decomposition from the definition of these spaces and thus does not require $X$ to have UMD (see Lemma \ref{multiplication2}).

The other two paraproducts are estimated in endpoint type Triebel-Lizorkin norms to the result
\begin{equation}\label{intro-esti-2}
 \|\Pi_i(m,f)\|_{F^{s}_{p,1}(\R^d,w_\gamma;X)} \leq C \|m\|_{B_{r,\infty}^{\frac{1+\mu}{r}}(\R, w_\mu)} \|f\|_{F^{s}_{p,\infty}(\R^d,w_\gamma;X)}, \qquad i=2,3,
\end{equation}
see the Lemmas \ref{multiplication3} and \ref{multiplication4}. As in \cite{Franke86} and \cite[Section 4.4]{RS96}, the proofs are based on Jawerth-Franke type embeddings and weighted estimates of series in spaces of entire analytic functions. These rather technical results are considered in detail in Appendix \ref{sec:analytic}.

Observe that in \eqref{intro-esti-2} there is a smoothing in the microscopic parameter $q$. Since
$$F^{s}_{p,1} \hookrightarrow F^{s}_{p,q} \hookrightarrow F^{s}_{p,\infty}, \qquad q\in [1,\infty],$$
on the left-hand side of \eqref{intro-esti-2} we have the smallest $F$-space and on the right-hand side of \eqref{intro-esti-2} we have the largest $F$-space for fixed $s$ and $p$. The smoothing can be employed for the $H$-spaces as follows: since
$$F^{s}_{p,1}(\R^d,w;X) \hookrightarrow H^{s,p}(\R^d,w;X)\hookrightarrow F^{s}_{p,\infty}(\R^d,w;X)$$
for arbitrary Banach spaces $X$ and weights $w\in A_p$ (see \cite{SchmSi05} and \cite[Proposition 3.12]{MeyVer1}), the estimate  \eqref{intro-esti-2} immediately gives
\begin{equation*}
 \|\Pi_i(m,f)\|_{H^{s,p}(\R^d,w_\gamma;X)} \leq C \|m\|_{B_{r,\infty}^{\frac{1+\mu}{r}}(\R, w_\mu)} \|f\|_{H^{s,p}(\R^d,w_\gamma;X)}, \qquad i=2,3.
\end{equation*}
In particular, the smoothing effect in \eqref{intro-esti-2} on the microscopic scale  allows to avoid the randomized Littlewood-Paley decomposition in the estimates of $\Pi_2$ and $\Pi_3$.

The idea to treat vector-valued $H$-spaces by considering the corresponding $F$-spaces and employing that many of their properties are independent of the microscopic parameter $q$ is due to Schmeisser $\&$ Sickel \cite{SchmSiunpublished} in the context of traces, see also \cite{MeyVer2, SSS}. \smallskip

This paper is organized as follows. In Section \ref{sec:prel} we introduce the weighted function spaces and in Section \ref{section:UMD} we consider the randomized Littlewood-Paley decomposition for weighted Bessel-potential spaces. The paraproducts are estimated in Section \ref{sec:Point}, and these results are applied in Section \ref{sec:p-mult} to obtain our main results on pointwise multiplication. In Appendix \ref{sec:analytic} we prove the required auxiliary results for spaces of entire analytic functions. \smallskip

\textbf{Notations.} Generic positive constants are denoted by $C$. For $x\in \R^d$ we write
$$x = (x',t), \qquad x'\in \R^{d-1}, \qquad t\in \R.$$
We let $\N = \{1, 2, 3, \ldots\}$ and $\N_0 = \N\cup\{0\}$.
 Throughout, $X$ and $Y$  are complex Banach spaces. It will explicitly be stated if further properties as UMD are assumed. The space of bounded linear operators from $X$ to $Y$ is denoted by $\calL(X,Y)$, and $\calL(X) = \calL(X,X)$. The Schwartz class is denoted by $\Schw(\R^d;X)$, and we write $\TD(\R^d;X) = \calL(\Schw(\R^d);X)$ for the $X$-valued tempered distributions. The Fourier transform is denoted by $\widehat{f}$ or $\F f$. For $\sigma = k + \sigma_*$ with  $k\in \N_0$ and $\sigma_*\in [0,1)$ we denote by $BC^\sigma(\R^d; X)$ the space of $C^k$-functions with bounded derivatives and $\sigma_*$-H\"older continuous $k$-th derivatives.

\section{Preliminaries\label{sec:prel}}
In this section we briefly recall some notions and facts from the Fourier analytic approach to function spaces (see \cite{Tri83}, and for the weighted case \cite{Bui82, HS08}). For the vector-valued setting we refer to \cite{SSS, SchmSiunpublished,Tr01} and \cite[Sections 2 and 3]{MeyVer1}.

\subsection{Muckenhoupt weights}\label{sec:Mucken}
A function $w:\R^d\to [0,\infty)$ is called a weight if $w\in L^1_{\text{loc}}(\R^d)$ and if it is positive almost everywhere on $\R^d$. For $p\in (1,\infty)$ the Muckenhoupt class of weights on $\R^d$ is denoted by $A_p$ or $A_p(\R^d)$, and $A_\infty = \bigcup_{p> 1} A_p$ (see \cite[Chapter 9]{GraModern} for the general theory). We are mainly interested in anisotropic power weights $w$ of the form
$$w_{\gamma}(x',t) = |t|^\gamma, \qquad  x=(x',t) \in \R^d, \qquad x'\in \R^{d-1},\qquad t \in \R.$$
This notation will be used throughout the rest of the paper. Here $w_{\gamma} \in A_p$ if and only if $\gamma \in (-1,p-1)$, see \cite[Example 1.5]{HS08}. For $w\in A_\infty$ the norm of $L^p(\R^d,w;X)$ is defined by
$$\|f\|_{L^p(\R^d,w;X)} = \left ( \int_{\R^d} \|f(x)\|_X^p w(x) \, dx\right)^{1/p}.$$
For $f\in L^1_{\text{loc}}(\R^d;X)$ the Hardy-Littlewood maximal operator $M$ is given by
$$(Mf)(x) = \sup_{r>0} \frac{1}{|B(x,r)|} \int_{B(x,r)} \|f(y)\|_X\,dy, \qquad x\in \R^d.$$
The operator $M$ is bounded on $L^p(\R^d,w;X)$ if and only if $w\in A_p$. More generally, the weighted Fefferman-Stein maximal inequality (see \cite[Theorem 3.1]{AJ80}, and also \cite[Proposition 2.2]{MeyVer1}) says that for $p\in (1,\infty)$, $q\in (1,\infty]$, $w\in A_p$ and any $(f_k)_{k\geq 0} \subset L^p(\R^d,w;\ell^q(X))$ we have
\begin{equation}\label{Fefferman-Stein}
\|(Mf_k)_{k\geq 0} \|_{L^p(\R^d,w;\ell^q)} \leq C\|(f_k)_{k\geq 0}\|_{L^p(\R^d,w;\ell^q(X))}.
\end{equation}
In Lemma \ref{lem:maxoperator} in the appendix we consider a version of this inequality for mixed-norm spaces.

\subsection{Weighted function spaces}\label{subsec:spaces}
Let $\Phi(\R^d)$ be the collection of all sequences $(\varphi_k)_{k\geq 0} \subset \Schw(\R^d)$ such that
\begin{align*}
\wh{\varphi}_0 = \wh{\varphi}, \qquad \wh{\varphi}_1(\xi) = \wh{\varphi}(\xi/2) - \wh{\varphi}(\xi), \qquad \wh{\varphi}_k(\xi) = \wh{\varphi}_1(2^{-k+1} \xi), \quad k\geq 2, \qquad \xi\in \R^d,
\end{align*}
with a generator function $\varphi$ of the form
\begin{equation*}
 0\leq \wh{\varphi}(\xi)\leq 1, \quad  \xi\in \R^d, \qquad  \wh{\varphi}(\xi) = 1 \ \text{ if } \ |\xi|\leq 1, \qquad  \wh{\varphi}(\xi)=0 \ \text{ if } \ |\xi|\geq \frac32.
\end{equation*}
Observe that $\supp \wh{\varphi}_k \subseteq \{2^{k-1} \leq |\xi|\leq \frac{3}{2}2^{k}\}$ for $k\geq 1$. For $(\varphi_k)_{k\geq 0} \in \Phi(\R^d)$ and $f\in \TD(\R^d;X)$ we set
$$S_k f = \varphi_k * f = \F^{-1} ( \wh{\varphi}_k \wh{f}).$$
The norms of the Besov space $B$, the Triebel-Lizorkin space $F$ and the Bessel-potential space $H$ are for $s\in \R$, $p\in (1,\infty)$, $q\in [1,\infty]$,   $w\in A_\infty$ and $f\in {\mathscr S}'(\R^d;X)$ given by
\[ \|f\|_{B_{p,q}^s (\R^d,w;X)} = \Big\| \big( 2^{sk}S_k f\big)_{k\geq 0} \Big\|_{\ell^q(L^p(\R^d,w;X))},\]
\[ \|f\|_{F_{p,q}^s (\R^d,w;X)} = \Big\| \big( 2^{sk}S_k f\big)_{k\ge 0} \Big\|_{L^p(\R^d,w;\ell^q(X))}, \]
\[\|f\|_{H^{s,p}(\R^d,w;X)} = \|\F^{-1} [(1+|\cdot|^2)^{s/2} \wh{f} ]\|_{L^p(\R^d,w;X)}.\]
Each choice of $(\varphi_k)_{k\geq 0} \in \Phi(\R^d)$ leads to an equivalent norm for the $B$- and $F$-spaces. For $m\in \N_0$ we also consider  Sobolev spaces $W$, with norm
$$\|f\|_{W^{m,p}(\R^d,w;X)} = \Big(\sum_{|\alpha|\leq m} \|D^{\alpha} f\|_{L^p(\R^d,w;X)}^p\Big)^{1/p}.$$
By \cite[Lemma 3.8]{MeyVer1}, the space $\Schw(\R^d;X)$ is dense in each of the above spaces if $q<\infty$. A useful substitute for the lack of density in case $q=\infty$ is the Fatou property. If $E = B_{p,q}^s (\R^d,w;X)$ or $E = F_{p,q}^s (\R^d,w;X)$, it says that for $(f_n)_{n\geq 0}\subset E$ we have
\begin{equation}\label{fatou}
 \lim_{n\to \infty} f_n = f \text{ in }\TD(\R^d;X), \quad \liminf_{n\to \infty}\|f_n\|_{E}<\infty \quad \Longrightarrow \quad f\in E, \quad \|f\|_E \leq \liminf_{n\to \infty} \|f_n\|_E,
\end{equation}
see \cite[Proposition 2.18]{SSS}.
We have the elementary embeddings
\begin{equation}\label{eq:scaleBF}
B_{p,\min\{p,q\}}^{s} (\R^d,w;X)\hookrightarrow F_{p,q}^{s} (\R^d,w;X) \hookrightarrow B_{p,\max\{p,q\}}^{s} (\R^d,w;X),
\end{equation}
and if $1\leq q_0\leq q_1\leq \infty$, then
\begin{equation}\label{eq:monotony}
 B_{p,q_0}^s (\R^d,w;X)\hookrightarrow B_{p,q_1}^s (\R^d,w;X), \qquad F_{p,q_0}^s (\R^d,w;X)\hookrightarrow F_{p,q_1}^s (\R^d,w;X).
\end{equation}
Moreover, for $w\in A_p$, $s\in  \R$ and $m\in \N_0$,
\begin{align}\label{eq:TriebelLizorkinH}
F^{s}_{p,1}(\R^d,w;X)  \hookrightarrow H^{s,p}(\R^d,w;X) \hookrightarrow F^{s}_{p,\infty}(\R^d,w;X),\\
\label{eq:TriebelLizorkinW}
F^{m}_{p,1}(\R^d,w;X)  \hookrightarrow W^{m,p}(\R^d,w;X) \hookrightarrow F^{m}_{p,\infty}(\R^d,w;X).
\end{align}
where the embeddings for $F^{s}_{p,1}$ and $F^{m}_{p,1}$ even hold in case $w\in A_\infty$.

\begin{remark}\label{HW}
Note that $L^p(\R^d,w;X) = H^{0,p}(\R^d,w;X) = W^{0,p}(\R^d,w;X)$. But $H^{1,p}(\R^d;X) = W^{1,p}(\R^d;X)$ if and only if $X$ has the UMD property (see \cite{McC84,Zim89}), and $L^p(\R^d;X) = F_{p,2}^0(\R^d;X)$ if and only if $X$ can be renormed as a Hilbert space (see \cite{HaMe96} and \cite[Remark 7]{SchmSiunpublished}).
\end{remark}

\subsection{A difference norm for weighted Besov spaces}
For an integer $m\geq 1$ define
\[\Delta_h^m f(x) = \sum_{l =0}^m {{m}\choose{l }} (-1)^l f(x+(m-l)h), \qquad x,h\in \R^d.\]
For $f\in L^p(\R^d,w;X)$ let
\[[f]_{B^s_{p,q}(\R^d,w;X)}^{(m)} = \Big(\int_0^\infty t^{-sq}  \Big\|t^{-d}\int_{|h|\leq t} \|\Delta_h^m f\|_X \, dh\Big\|_{L^p(\R^d,w)}^{q} \, \frac{dt}{t}\Big)^{1/q},\]
with the usual modification if $q=\infty$, and set
\[\nn f\nn_{B^s_{p,q}(\R^d,w;X)}^{(m)} = \|f\|_{L^p(\R^d,w;X)} + [f]_{B^s_{p,q}(\R^d,w;X)}^{(m)}.\]

One can extend a well-known result on the equivalence of norms to the weighted case
(cf. \cite{SchmSiunpublished}, \cite[Section 2.5.10]{Tri83} and \cite[Theorem 6.9]{Triebel3}). A similar result for weighted $F$-spaces is stated  in \cite[Proposition 2.3]{MeyVer2}.

\begin{proposition}\label{prop:Lpsmoothness-Besov} Let $s>0$, $p\in (1, \infty)$, $q\in [1, \infty]$ and $w\in A_p$. Let $m \in \N$ be such that $m>s$. There is a constant $C>0$ such that for all $f\in L^p(\R^d,w;X)$ one has
\begin{equation}\label{eq:equivnorm-Besov}
C^{-1} \|f\|_{B^s_{p,q}(\R^d,w;X)}\leq \nn f\nn_{B^s_{p,q}(\R^d,w;X)}^{(m)} \leq C\|f\|_{B^s_{p,q}(\R^d,w;X)},
\end{equation}
whenever one of these expressions is finite.
\end{proposition}

It is often more convenient to work with  the $L^p(\R^d,w;X)$-modulus of smoothness, defined by
\[\omega_{p,w}^m(f,t) = \sup_{|h|\leq t} \|\Delta^m_h f\|_{L^p(\R^d,w;X)}, \qquad t>0.\]
In the unweighted case $w\equiv 1$, for any integer $m> s$ the expression
\[\| f\|_{B^{s}_{p,q}(\R^d,w;X)}^{(m)} = \|f\|_{L^p(\R^d,w;X)} + \Big(\int_0^\infty t^{-sq} \omega_{p,w}^m(f,t)^q \, \frac{dt}{t}\Big)^{1/q},\]
defines an equivalent norm on $B^{s}_{p,q}(\R^d,w;X)$ (modification if $q=\infty$). We do not know if this extends to the weighted setting. However, by Minkowski's inequality one has
\[\Big\|t^{-d} \int_{|h|\leq t} \|\Delta_h^m f\|_X \, dh\Big\|_{L^p(\R^d,w)} \leq t^{-d} \int_{|h|\leq t} \|\Delta_h^m f\|_{L^p(\R^d,w;X)} \, dh \leq C \sup_{|h|\leq t}\|\Delta_h^m f\|_{L^p(\R^d,w;X)}.\]
Therefore, one always has
\begin{equation}\label{remark-besov-norm}
 \nn f\nn_{B^{s}_{p,q}(\R^d,w;X)}^{(m)} \leq C \| f\|_{B^{s}_{p,q}(\R^d,w;X)}^{(m)}.
\end{equation}

\section{UMD-valued Bessel-potential spaces}\label{section:UMD}

In this section we derive a Littlewood-Paley decomposition for the spaces $H^{s,p}(\R^d,w;X)$, where $X$ has UMD and $w\in A_p$. As preparations we first recall some notions in this context and record a Mihlin multiplier theorem for $L^p(\R^d,w;X)$, which follows from the results of \cite{HH12}. We then give a first multiplication estimate for H\"older continuous functions and $H^{s,p}(\R^d,w;X)$, which is based on bilinear complex interpolation.

\subsection{UMD spaces, Rademacher functions and $\mathcal R$-boundedness}\label{subsec:UMD}
A Banach space $X$ is said to have UMD if for any probability space $(\Omega,\mathscr{A},\P)$ and $p\in (1, \infty)$ martingale differences are unconditional in $L^p(\Omega;X)$ (see \cite{Ama95, Bu3, RF} for a survey on the subject). The UMD property of a Banach space turns out to be equivalent to the boundedness of the vector-valued extension of the Hilbert transform on $L^p(\R;X)$. For this reason UMD is sometimes also called of class $\mathcal{H} \mathcal{T}$. Many other Fourier multipliers are known to be bounded in $L^p(\R^d;X)$ and in particular, the classical Mihlin Fourier multiplier theorem holds in the vector-valued setting if and only if $X$ has UMD, see \cite{Bou2,McC84,Zim89} (and Proposition \ref{prop:weightedmihlin} below).

Let us mention a few facts on UMD spaces (see \cite[Section III.4]{Ama95}).
\begin{enumerate}[(a)]
\item Hilbert spaces have UMD.
\item Closed subspaces and the dual of UMD spaces have UMD.
\item If $X$ has UMD, then $L^p(\Omega; X)$ has UMD for each $\sigma$-finite measure space $\Omega$ and $p\in (1,\infty)$.
\item The reflexive range of the classical function spaces such as $L^p$, $H^{s,p}$, $B_{p,q}^s$, $F_{p,q}^s$ have UMD.
\item UMD spaces are reflexive. Hence $L^1$, $\ell^1$, $L^\infty$, $C([0,1])$ and $c_0$ do not have UMD.
\end{enumerate}

 A sequence of random variables $(r_k)_{k\geq 0}$ on $\Omega$ is called a Rademacher sequence if $\P(\{r_k = 1\}) = \P(\{r_k=-1\}) = 1/2$ for $k\geq 0$ and $(r_k)_{k\geq 0}$ are independent. For instance, one can take $\Omega = (0,1)$ with the Lebesgue measure and $r_k(\omega) = \text{sign}[\sin(2^{k+1}\pi \omega)]$ for $\omega\in \Omega$.

A family of operators $\mathcal{T} \subset \calL(X,Y)$ is called $\mathcal R$-bounded, if some $p\in [1,\infty)$ there is a constant $C_p$ such that for all $N\geq 1$, for all $T_0,..., T_N \in \mathcal T$ and all $x_0,...,x_N\in X$ it holds that
$$\Big \|\sum_{k=0}^N r_k T_k x_k \Big\|_{L^p(\Omega;Y)} \leq C_p \Big\|\sum_{k=0}^N r_k  x_k \Big \|_{L^p(\Omega;X)}.$$
The infimum of all constants $C_p$ satisfying the above estimate is denoted by $\mathcal R_p(\mathcal T)$ and is called the $\mathcal R_p$-bound of $\mathcal T$. One can show that if the inequality is satisfied for one $p$, then it holds for all $p$. We often neglect the dependence of the $\mathcal R$-bound on $p$.
For further information on $\mathcal R$-boundedness we refer to \cite{DHP, KuWe}.

\subsection{Fourier multipliers}
For a symbol $m\in L^\infty(\R^d)$ we define the operator $T_m$ by
$$T_m: \Schw(\R^d;X) \to \TD(\R^d;X), \qquad T_m f = \mathcal \F^{-1} (m \wh f).$$
For $p\in [1,\infty)$ and $w\in A_\infty$ the Schwartz class $\Schw(\R^d;X)$ is dense in $L^p(\R^d;X)$, see \cite[Lemma 3.8]{MeyVer1}. The following Mihlin type multiplier theorem provides a sufficient condition for the boundedness of $T_m$. It is a simple consequence of \cite[Corollary 2.10]{HH12}. For the scalar case $X = \C$ we refer to \cite[Section IV.3]{GCRdF}. A version with operator-valued multiplier holds as well. For this one needs an $\mathcal R$-boundedness version of the condition \eqref{Mihlin-assum} (see \cite[Theorems 3.6 and 3.7]{HHN}, \cite[Theorem 4.4]{StrWe} and \cite{We}).

\begin{proposition}\label{prop:weightedmihlin}
Let $X$ have \emph{UMD}, $p\in (1, \infty)$ and $w\in A_p$. Assume that $m\in C^{d+2}(\R^d\setminus\{0\})$ satisfies
\begin{equation}\label{Mihlin-assum}
C_m = \sup_{|\alpha| \leq d+2}\sup_{\xi \neq 0} |\xi|^{|\alpha|} |D^{\alpha} m(\xi)| <\infty.
\end{equation}
Then $T_m$ extends to a bounded operator on $L^p(\R^d,w;X)$, and its operator norm only depends on $d$, $X$, $p$, $w$ and $C_m$.
\end{proposition}
\begin{proof}
By \cite[Corollary 2.10]{HH12} we have to verify that $T_m$ is a vector-valued Calder\'on-Zygmund operator, in the sense of \cite[Definition 2.6]{HH12}. Assumption \eqref{Mihlin-assum} for $\alpha\leq (1,\ldots,1)$ implies that $T_m$ belongs to $\calL(L^p(\R^d;X))$, see \cite[Proposition 3]{Zim89}. Further, $\F^{-1}m$ may be represented by a function $K\in C^1(\R^d\setminus\{0\})$ satisfying $|K(x)|\leq C|x|^{-d}$ and $|\nabla K(x)|\leq C|x|^{-(d+1)}$ for $x\neq 0$, see the proof of \cite[Proposition VI.4.4.2]{Stein93}. Hence $T_m$ is represented by the convolution with a singular kernel. We conclude that \cite[Corollary 2.10]{HH12} applies to $T_m$.
\end{proof}

\subsection{Equivalent norms and Littlewood-Paley theory}

The following characterizations can be deduced from Proposition \ref{prop:weightedmihlin}. We fix a Rademacher sequence $(r_k)_{k\geq 0}$ on a probability space $\Omega$, and a further a sequence $(\varphi_k)_{k\geq 0} \in \Phi(\R^d)$. Recall that $S_k f  = \varphi_k*f$.

\begin{proposition}\label{prop:UMDHisF}
Let $X$ have \emph{UMD}, $p\in (1, \infty)$  and $w\in A_p$. Then
\begin{equation}
 H^{m,p}(\R^d,w;X) = W^{m,p}(\R^d,w;X) \qquad  \text{ for all }\, m \in \N_0. \label{H=W}
\end{equation}
Moreover, for $s\in \R$ we have that $f\in \TD(\R^d;X)$ belongs to $H^{s,p}(\R^d,w;X)$ if and only if
\[\sup_{n\geq 0}\Big \| \sum_{k=0}^n r_k 2^{sk} S_k f\Big \|_{L^p(\O;L^p(\R^d,w;X))}<\infty.\]
In this case the series $\sum_{k\geq 0} r_k 2^{sk} S_k f$ converges in $L^p(\Omega; L^p(\R^d,w;X))$, and
\begin{equation}\label{eq:convergenceFrad}
\|f\|_{F^{s}_{p,\Rad}(\R^d,w;X)}= \Big\| \sum_{k\geq 0} r_k 2^{sk} S_k f\Big\|_{L^p(\O;L^p(\R^d,w;X))} = \sup_{n\geq 0}\Big \| \sum_{k=0}^n r_k 2^{sk} S_k f\Big \|_{L^p(\O;L^p(\R^d,w;X))}
\end{equation}
defines an equivalent norm on $H^{s,p}(\R^d,w;X)$.
\end{proposition}

\begin{remark} \
\begin{enumerate}[(i)]
 \item For $H^{1,p}(\R^d;X) = W^{1,p}(\R^d;X)$ it is necessary that $X$ has UMD, see Remark \ref{HW}.
 \item The extended real number $\|f\|_{F^{s}_{p,\Rad}(\R^d,w;X)}$ is well-defined for every tempered distribution $f$ and therefore one could study the space $F^{s}_{p,\Rad}(\R^d,w;X)$ on its own, see \cite{Ver12}. The result shows that if $X$ has UMD, then $F^{s}_{p,\Rad}$ coincides with $H^{s,p}$. In particular, for $w\in A_p$ in the scalar case one has
$$\|f\|_{F^s_{p,2}(\R^d,w)}\eqsim \|f\|_{F^{s}_{p,\Rad}(\R^d,w)}.$$
The identity $F^s_{p,2}(\R^d,w) = H^{s,p}(\R^d,w)$ was proved in \cite{Ry01} for weights $w$ which satisfy only a local $A_p$-condition.
\end{enumerate}

\end{remark}

\begin{proof}[Proof of Proposition \ref{prop:UMDHisF}] \emph{Step 1.}
Using Proposition \ref{prop:weightedmihlin}, the identity (\ref{H=W}) can be shown as in the unweighted scalar case (see \cite[Theorem 6.2.3]{BeLo} or \cite[Section 2.3.3]{Tr1}).

\emph{Step 2.} Assume $\|f\|_{F^{s}_{p,\Rad}(\R^d,w;X)}<\infty$. Since closed subspaces of UMD spaces have UMD and the sequence space $c_0$ does not have UMD, it follows that $X$ does not contain a copy of $c_0$. We therefore conclude from \cite[Theorem 9.29]{LeTa} that the series $\sum_{k=0}^\infty r_k 2^{sk} S_k f$ converges in $L^p(\O;L^p(\R^d,w;X))$. It follows from the properties of the Rademacher functions that
\[\Big \|\sum_{k=0}^n r_k 2^{sk} S_k f\Big\|_{L^p(\O;L^p(\R^d,w;X))}\leq \Big\|\sum_{k=0}^\infty r_k 2^{sk} S_k f\Big\|_{L^p(\O;L^p(\R^d,w;X))},\]
which implies one inequality for the assertion in \eqref{eq:convergenceFrad}. The other inequality is trivial.

\emph{Step 3.} Let  $f\in H^{s,p}(\R^d,w;X)$ and write $f_s = \F^{-1}[(1+|\cdot|^2)^{s/2} \wh{f}]\in L^p(\R^d,w;X)$. Fix $n\geq 0$, $\omega\in \O$ and define the scalar symbol $m_n \in C^\infty(\R^d)$ by
\[m_n(\xi) = \sum_{k=0}^n r_k(\omega) 2^{sk}(1+|\xi|^2)^{-s/2} \widehat{\varphi}_k(\xi).\]
For each $\xi\in \R^d$, here at most three summands  are nonzero. Since $\widehat{\varphi}_k$ is supported around $|\xi| = 2^k$ and $\|D^\beta \wh\varphi_k\|_\infty\leq C_\beta 2^{-k|\beta|}$, it follows that
$$C_m = \sup_{n\geq 0}\sup_{|\alpha|\leq d+2} \sup_{\xi\neq 0}|\xi|^{|\alpha|} |D^{\alpha} m_n(\xi)|<\infty,$$
where $C_m$ is independent of $\omega$.
By Proposition \ref{prop:weightedmihlin}, the corresponding operators $T_{m_n}$ are bounded on $L^p(\R^d,w;X)$, uniformly in $n$ and $\omega$. From this we obtain
\begin{align*}
\Big\|\sum_{k=0}^n r_k(\omega) 2^{sk} \varphi_k * f\Big\|_{L^p(\R^d,w;X)} &= \|T_{m_n}
f_s\|_{L^p(\R^d,w;X)} \\ & \leq C \|f_s\|_{L^p(\R^d,w;X)} = C \|
f\|_{H^{s,p}(\R^d,w;X)}.
\end{align*}
Taking the $L^p(\O)$-norm and the supremum over $n$ yields $\|f\|_{F^{s}_{p,\Rad}(\R^d,w;X)}\leq C \|
f\|_{H^{s,p}(\R^d,w;X)}$.

\emph{Step 4.} For the converse estimate, assume that $\|f\|_{F^{s}_{p,\Rad}(\R^d,w;X)}<\infty$. As we have seen in Step 2, then $\sum_{k\geq 0} r_k 2^{sk} \varphi_k * f$ converges in $L^p(\O;L^p(\R^d,w;X))$. From \cite[Theorem 2.4]{LeTa} we get that $\sum_{k\geq 0} r_k(\omega) 2^{sk} \varphi_k * f$ converges in $L^p(\R^d,w;X)$ for almost every $\omega\in \Omega$. Choose $(\widehat{\psi}_k)_{k \geq 0}$ such that $0\leq \wh{\psi}_k\leq 1$, $\wh{\psi}_k = 1$ on $\text{supp}\, \wh{\varphi}_k$, $\text{supp}\,\wh{\psi}_0\subset \{0 \leq |\xi|\leq 2\}$ and $\text{supp}\, \wh{\psi}_k\subset \{2^{k-2}\leq |\xi|\leq 2^{k+1}\}$ for $k\geq 1$. For $\omega\in \Omega$ we set
$$ m_\omega = \sum_{l\geq 0} r_l(\omega) 2^{-sl} (1+|\cdot|^2)^{s/2} \wh{\psi}_l,\qquad g_\omega = \sum_{k\geq 0} r_k(\omega) 2^{sk} \varphi_k * f.$$
Let $f_s$ be as in Step 3. Then the independence and symmetry of the Rademacher random variables together with the support conditions on $\wh \varphi_k, \wh \psi_k$ imply that $f_s = \int_\Omega T_{m_\omega} g_\omega \, d\P(\omega)$. As before,
$$ C_m =  \sup_{|\alpha|\leq d+2} \sup_{\xi\neq 0}|\xi|^{|\alpha|} |D^{\alpha} m_\omega(\xi)|<\infty$$
is independent of $\omega$. Thus  $\|T_{m_\omega} g_\omega\|_{L^p(\R^d,w;X)}\leq C \|g_\omega\|_{L^p(\R^d,w;X)}$ for almost every $\omega$ by Proposition \ref{prop:weightedmihlin}.  Therefore, using also Jensen's inequality and Fubini's theorem,
\begin{align*}
\|f\|_{H^{s,p}(\R^d,w;X)}^p &\, =\|f_s\|_{L^p(\R^d,w;X)}^p =  \Big \|\int_\Omega T_{m_\omega} g_\omega \, d\P(\omega)\Big \|_{L^p(\R^d,w;X)}^p\\
&\, \leq \int_\Omega \big \| T_{m_\omega} g_\omega \big \|_{L^p(\R^d,w;X)}^p  \, d\P(\omega) \leq  C \int_\Omega \big \|  g_\omega \big \|_{L^p(\R^d,w;X)}^p \, d\P(\omega) = \|f\|_{F^{s}_{p,\Rad}(\R^d,w;X)}^p.
\end{align*}
Hence $f\in H^{s,p}(\R^d,w;X)$ and the required estimate follow.\end{proof}

Another equivalent norm for UMD-valued $H$-spaces is given as follows.

\begin{proposition}\label{prop:differentiation2}
Let $X$ have \emph{UMD}, $s\in \R$, $p\in (1, \infty)$ and $w\in A_{p}$. Then for each $m\in \N$,
\begin{equation}\label{eq:Besseldifferentiation}
\sum_{|\alpha|\leq m} \|D^\alpha f\|_{H^{s-m,p}(\R^d,w;X)}
\end{equation}
defines an equivalent norm on $H^{s,p}(\R^d,w;X)$
\end{proposition}
\begin{proof}
This is a consequence of (\ref{H=W}) and the fact that $D^{\alpha}$ and the Bessel-potential commute on $\TD(\R^d;X)$.
\end{proof}

\subsection{Duality, functional calculus and complex interpolation} Let $X$ be a Banach space such that its dual space $X^*$ has the Radon-Nikodym property RNP, cf.  \cite[Definition III.1/3]{DU77}. For instance, reflexive Banach spaces and thus UMD spaces have RNP, see \cite[Corollary III.2/12]{DU77}.

If $X^*$ has RNP then it follows from \cite[Theorem IV.1/1]{DU77} that for a $\sigma$-finite measure space $(S,\Sigma,\mu)$ and $p\in (1,\infty)$ with dual exponent $p' = \frac{p}{p-1}$ one has $L^p(S,\mu;X)^* = L^{p'}(S,\mu;X^*)$, induced by the pairing $\int_{S} \lb f(x), g(x)\rb_{X, X^*} d\mu$.

Since this pairing does not respect the $A_p$-classes, in the context of weights it is more convenient to work with
$$\lb f,g\rb = \int_{\R^d} \lb f(x),  g(x)\rb_{X,X^*} \,dx.$$
Recall from \cite{GraModern} that for $w\in A_p$ the dual weight $w' = w^{-\frac{1}{p-1}}$ with respect to $p$ belongs to $A_{p'}$.
\begin{proposition} \label{dual-H} Let $X$ be a Banach space such that $X^*$ has \emph{RNP}, let $s\in \R$, $p\in (1,\infty)$ and let $w\in A_p$. Then
$$|\lb f,g\rb| \leq \|f\|_{H^{s,p}(\R^d,w;X)} \|g\|_{H^{-s,p'}(\R^d,w';X^*)}, \qquad f\in \Schw(\R^d;X), \quad g\in \Schw(\R^d,X^*),$$
such that the pairing $\lb \cdot,\cdot\rb$ extends continuously to $H^{s,p}(\R^d,w;X) \times H^{-s,p'}(\R^d,w';X^*)$. Every element of $H^{s,p}(\R^d,w;X)^*$ is of the form $\lb \cdot, g\rb$ with $g\in H^{-s,p'}(\R^d,w';X^*)$. In this sense, $$H^{s,p}(\R^d,w;X)^* = H^{-s,p'}(\R^d,w';X^*).$$
\end{proposition}
\begin{proof}  For $s=0$, the weighted case can easily be deduced from the unweighted case. For general $s\in \R$ we have $\lb J_s f,g\rb = \lb f, J_s g\rb$, such that the same arguments as in \cite[Theorem 9]{Cal61} for the unweighted scalar case apply.
\end{proof}

To prove that UMD-valued $H$-spaces form a complex interpolation scale we record  the following result on bounded $\mathcal H^\infty$-calculi. For a definition and the properties of this functional calculus we refer to \cite{DHP, KuWe}.

\begin{proposition}\label{prop:Hinfty}
Let $X$ have \emph{UMD}, $p\in (1, \infty)$ and $w\in A_p$. The following assertions hold true.
\begin{enumerate}
\item[\emph{(a)}] The operator $\partial_t$ with domain $H^{1,p}(\R,w;X)$ on $L^p(\R,w;X)$ has a bounded $\mathcal H^\infty$-calculus of angle $\frac{\pi}{2}$.
\item[\emph{(b)}] The operator $-\Delta$ with domain $H^{2,p}(\R^d,w;X)$ on $L^p(\R^d,w;X)$ has a bounded $\mathcal H^\infty$-calculus of angle zero.
\end{enumerate}
\end{proposition}
\begin{proof}
Using Proposition \ref{prop:weightedmihlin}, one can argue  as in \cite[Example 10.2]{KuWe}.
\end{proof}

The complex interpolation functor is denoted by $[\cdot,\cdot]_\theta$. We refer to \cite[Proposition 6.1]{MeyVer1} for real interpolation of vector-valued $H$-spaces.

\begin{proposition}\label{interpol-complex}
Let $X$ have \emph{UMD}, $p\in (1, \infty)$ and $w\in A_p$. Assume $s_0 < s_1$, $\theta\in (0,1)$ and $s = (1-\theta)s_0 + \theta s_1$. Then
\[[H^{s_0,p}(\R^d,w;X), H^{s_1,p}(\R^d,w;X)]_{\theta} = H^{s,p}(\R^d,w;X).\]
\end{proposition}
\begin{proof}
By Proposition \ref{prop:Hinfty}, the operator $1-\Delta$ with domain $D(A) = H^{2,p}(\R^d,w;X)$ on $L^p(\R^d,w;X)$ has a bounded $\mathcal H^\infty$-calculus of angle zero. This also implies the boundedness of its imaginary powers. Since $(1-\Delta)^{s_0/2}$ commutes with $1-\Delta$, the same is true for the realization of $A_{s_0}$ of $1-\Delta$ on $H^{s_0,p}(\R^d,w;X)$. Therefore, by  \cite[Theorem 1.15.3]{Tr1},
$$[H^{s_0,p}(\R^d,w;X), D(A_{s_0}^{(s_1-s_0)/2})]_{\theta} = D(A_{s_0}^{\theta(s_1-s_0)/2}).$$
Since  $D(A_{s_0}^{\tau/2}) = H^{s_0+\tau,p}(\R^d,w;X)$ for any $\tau > 0$, the assertion follows.\end{proof}

\subsection{Multiplication by H\"older continuous functions} Using bilinear interpolation, we give a first result on pointwise multiplication. An analogous result for $F$- and $B$-spaces is obtained in Proposition \ref{prop:mult-smooth}. For $s <0$ the product is interpreted as an extension via density from the usual pointwise product of smooth functions.

\begin{proposition} \label{thm:mult-smooth2}
Let $X$ and $Y$ have \emph{UMD}, $s\in \R$, $p\in (1,\infty)$ and $w\in A_p$. Assume $\sigma > |s|$. Then
$$\|mf\|_{H^{s,p}(\R^d,w;Y)} \leq C\|m\|_{BC^{\sigma}(\R^d; \calL(X,Y))} \|f\|_{H^{s,p}(\R^d,w;X)}.$$
\end{proposition}
\begin{proof}
By \eqref{H=W}, the result for $s \in \N_0$ follows immediately from Leibniz' formula.  For noninteger $s > 0$ it follows from the integer case and bilinear complex interpolation, see \cite[Theorem 4.4.1]{BeLo}. Here the $H$-spaces are interpolated with Proposition \ref{interpol-complex}. For the interpolation of the $BC^m$-spaces with $m\in \N_0$ we note that for $\theta\in (0,1)$ and $\varepsilon > 0$ one has
$$BC^{m+\theta+\varepsilon} \hookrightarrow B_{\infty,1}^{m+\theta} = [B_{\infty,1}^m, B_{\infty,1}^{m+1}]_{\theta} \hookrightarrow [BC^m, BC^{m+1}]_\theta,$$
see the Sections 2.4.7 and 2.5.7 of  \cite{Tri83} for the scalar case.

Let finally $s<0$. Then for $f\in H^{s,p}(\R^d,w;X)$ and $g\in H^{-s,p'}(\R^d,w';Y^*)$ one has
\begin{align*}
|\lb mf, g\rb| &= |\lb f, m^* g\rb|\leq \|f\|_{H^{s,p}(\R^d;X)} \|m^* g\|_{H^{-s,p'}(\R^d,w';X^*)}
\\ & \leq C \|f\|_{H^{s,p}(\R^d;X)} \|m^*\|_{BC^{\sigma}(\R^d; \calL(X,Y))}  \|g\|_{H^{-s,p'}(\R^d,w';Y^*)}
\end{align*}
Taking the supremum over all $g$ with norm smaller than one and recalling that $\|m(x)\|_{\calL(X,Y)} = \|m(x)^*\|_{\calL(Y^*,X^*)}$, the required estimate follows from Proposition \ref{dual-H}.
\end{proof}

The above result also holds with $H^{s,p}$ replaced by $B^{s}_{p,q}$, general $X$ and $Y$ and $s>0$. This follows from the $W^{m,p}$-case and real interpolation with parameter $q$. For reflexive spaces the case $s<0$ can be obtained by duality under restrictions on the parameters $p$ and $q$.

\section{Estimates of paraproducts\label{sec:Point}}

To investigate pointwise multipliers we follow \cite{Franke86, RS96, Tri83} and use the decomposition of a product into paraproducts. The basis for their estimates and convergence are the results in Appendix \ref{sec:analytic} on weighted spaces of entire analytic functions.

\subsection{Preliminaries}
We fix a Rademacher sequence $(r_k)_{k\geq 0}$ on a probability space $\Omega$ and a sequence $(\varphi_k)_{k\geq 0} \in \Phi(\R^d)$ with the corresponding operators $S_k f = \varphi_k * f$.

\begin{lemma} \label{Sk-Rbounded}
Let $X$ have \emph{UMD}, $p\in (1,\infty)$ and $w\in A_p$.  Then $(S_k)_{k\geq 0}$ is an $\mathcal R$-bounded subset of $\calL (L^p(\R^d,w;X))$.
\end{lemma}

\begin{proof}
Let $(r'_l)_{l\geq 0}$ be an independent copy of $(r_k)_{k\geq 0}$ on $\Omega' = \Omega$. For any Banach space $Y$, as in \cite[Lemma 3.12]{GW03} one can prove that for $(y_{kl})_{k,l=0}^N \subset Y$  one has
\begin{equation}\label{eq:diagonalbdd}
\Big\|\sum_{k=0}^N r_k y_{kk}\Big\|_{L^p(\O;Y)}\leq \Big\|\sum_{k,l=0}^N r_k r_l' y_{kl}\Big\|_{L^p(\O\times\O';Y)}
\end{equation}
Now let $f_0,..., f_N \in L^p(\R^d,w;X)$. Since $X$ has UMD, this is also true for $X_\Omega = L^p(\Omega;X)$, see \cite[Theorem III.4.5.2]{Ama95}. Using \eqref{eq:diagonalbdd} with $y_{kl} = S_l f_k$ on $Y =L^p(\R^d,w;X)$ and Proposition \ref{prop:UMDHisF} with $s = 0$ on $L^p(\R^d,w; X_\Omega)$ we obtain
\begin{align*}
\Big \|\sum_{k=0}^N r_k S_k f_k\Big\|_{L^p(\Omega; L^p(\R^d,w;X))} & \leq \Big \|\sum_{k,l=0}^N r_k r_l' S_l f_k\Big\|_{L^p(\O'\times \Omega; L^p(\R^d,w;X))}
\\ & = \Big \|\sum_{l=0}^N r_l' S_l \Big(\sum_{k=0}^N r_k f_k\Big)\Big\|_{L^p(\O'; L^p(\R^d,w;X_\Omega))}
\\ & \leq C \Big \|\sum_{k=0}^N r_k f_k\Big\|_{L^p(\O; L^p(\R^d,w;X))}.
\end{align*}
This shows the $\mathcal R$-boundedness of $(S_k)_{k\geq 0}$.
\end{proof}

On $\TD(\R^d;X)$ we define the operators
$$S^l := \sum_{k=0}^l S_k, \quad l\in \N_0, \qquad S^{-l} :=0, \quad l\in \N.$$
Since $\wh\varphi_k = \wh\varphi_0(2^{-k}\cdot) - \wh\varphi_0(2^{-k+1}\cdot)$ for $k\geq 1$, we have $S^l f = \F^{-1} (\wh \varphi_0(2^{-l}\cdot)\wh f)$  and thus
$$S^l f \to f \qquad  \text{in }\TD(\R^d;X)\;\;\text{ as }l\to\infty.$$

The next result is useful for operator-valued pointwise multipliers on $H$-spaces.
\begin{lemma} \label{Slm-Rbounded}Let $X$ and $Y$ be Banach spaces. Let  $m:\R^d\to \calL(X,Y)$ be strongly measurable and assume that the image of $m$ is $\mathcal R$-bounded by $\mathcal R(m)$. Then $\mathcal M = \{(S^lm)(x)\,:\, l\in \N_0,\; x\in \R^d\}$ is $\mathcal R$-bounded in $\calL(X,Y)$ with $\mathcal R(\mathcal M) \leq 2\|\varphi_0\|_{L^1(\R^d)} \mathcal R(m)$.
\end{lemma}
\begin{proof} For all $l$ and $x$ we have
$$(S^lm)(x) = \int_{\R^d} 2^{ld}\varphi_0(2^l(x-y))m(y)\, dy, \qquad \|2^{ld}\varphi_0(2^l(x-\cdot))\|_{L^1(\R^d)} = \|\varphi_0\|_{L^1(\R^d)}.$$
Thus the result follows from \cite[Corollary 2.14]{KuWe}.
\end{proof}

The following simple fact is analogous to \cite[Lemma 4.4.2]{RS96}. We consider the mixed-norm spaces
$$L^{p(r)}(\R^d,w;X) = L^p(\R^{d-1}; L^r(\R,w;X)),$$
for a weight $w\in A_\infty(\R)$ depending only on the last coordinate $t$. See also Appendix \ref{sec:analytic}.
\begin{lemma}\label{lem:para1}
Let  $s<0$, $p,r\in (1,\infty)$, $q\in [1,\infty]$ and $w\in A_\infty(\R)$. Then for all $f\in \TD(\R^d;X)$ one has
$$\|(2^{sl} S^{l} f)_{l\geq 0} \|_{\ell^{q}(L^{p(r)}(\R^d,w;X))} \leq C \|(2^{sk} S_{k} f)_{k\geq 0}\|_{\ell^{q}(L^{p(r)}(\R^d,w;X))}.$$
\end{lemma}
\begin{proof} We consider $q<\infty$, the case $q=\infty$ is analogous.
Writing $Y=L^{p(r)}(\R^d,w;X)$, it follows from Young's inequality for discrete convolutions that
\begin{align*}
\|(2^{sl} S^{l} f)_{l\geq 0} \|_{\ell^{q}(Y)} &\, \leq \Big(\sum_{l=0}^\infty \Big( \sum_{k=0}^l 2^{s(l-k)} 2^{sk}\|S_k f\|_Y \Big)^q\Big)^{1/q}
\\ &\, \leq \Big( \sum_{l=0}^\infty 2^{s l} \Big)\Big(\sum_{k=0}^\infty  2^{sk}\|S_k f\|_Y^q \Big)^{1/q}
\leq C \|(2^{sk} S_{k} f)_{k\geq 0}\|_{\ell^{q}(Y)},
\end{align*}
where $C= \sum_{l\geq 0} 2^{sl}$ is finite by the assumption $s < 0$.
\end{proof}

\subsection{Paraproducts}\label{sec:paraproducts}

Let $X,Y$ be Banach spaces. As in \cite[Section 4.2]{RS96} we define the product $mf \in \TD(\R^d;Y)$ of $m\in \TD(\R^d; \calL(X,Y))$ and $f\in \TD(\R^d;X)$ by
$$mf = \lim_{l\to\infty} S^l m \cdot S^l f,$$
provided this limit exists in $\TD(\R^d;Y)$. If one factor is smooth with bounded derivatives or if $m\in L^r$ and $f\in L^{r'}$, then this definition yields the usual product of a function and a distribution or the pointwise product of functions, respectively (see \cite[Section 4.2.1]{RS96}).

As in \cite[Section 4.4]{RS96}, if the paraproducts
$$\Pi_1(m,f) = \sum_{k=2}^\infty (S^{k-2}m) (S_k f),
\qquad
\Pi_2(m,f) = \sum_{k=0}^\infty\sum_{j=-1}^1  (S_{k+j}m) (S_k f),$$
$$\Pi_3(m,f)= \sum_{k=2}^\infty (S_k m) (S^{k-2} f),$$
exist in $\TD(\R^d;Y)$, then $mf$ exists as well and one has
$$mf = \Pi_1(m,f) + \Pi_2(m,f) + \Pi_3(m,f).$$
Since $\supp \wh{\varphi}_k \subset \{2^{k-1} \leq |\xi|\leq \frac{3}{2}2^{k}\}$ for $k\geq 1$, for the Fourier supports of the summands we have
\begin{equation}\label{FsupportPi2}
 \supp \F  [(S_{k+j}m) (S_k f)] \subset \{  |\xi|\leq 5\cdot 2^{k}\}, \qquad k\geq 0, \quad j\in \{-1,0,1\},
\end{equation}
\begin{equation}\label{FsupportPi3}
 \supp \F [(S_k m) (S^{k-2} f)] \subset \{ 2^{k-3} \leq |\xi|\leq 2^{k+1}\}, \qquad k \geq 2.
 \end{equation}

\subsection{Estimates of $\Pi_1$}

The paraproducts are estimated in different ways. We start with $\Pi_1$. For the Bessel-potential spaces we use the Littlewood-Paley decomposition from Proposition \ref{prop:UMDHisF} and therefore require $X$ and $Y$ to have UMD.

\begin{lemma} \label{multiplication1} Let $X$ and $Y$ have \emph{UMD}, $s\in \R$, $p\in (1,\infty)$ and $w\in A_p$. Let $m:\R^d \to \calL(X,Y)$ be strongly measurable and assume that the image of $m$ is $\mathcal R$-bounded by $\mathcal R(m)$. Then for all $f\in H^{s,p}(\R^d,w;Y)$ the limit $\Pi_1(m,f)$ exists in $\TD(\R^d;Y)$  and
$$\|\Pi_1(m,f)\|_{H^{s,p}(\R^d,w;Y)} \leq C \mathcal R(m)\|f\|_{H^{s,p}(\R^d,w;X)}.$$
\end{lemma}

\begin{remark}\label{rem:R-bound-scalar}
 If $m$ is scalar-valued we have $\mathcal R(m) \leq 2\|m\|_\infty$, see \cite[Proposition 2.5]{KuWe}. So in this case the assumptions on $m$ reduce to $m\in L^\infty(\R^d)$.
\end{remark}

\begin{proof}[Proof of Lemma \ref{multiplication1}]
We write  $ \Pi_1(m,f) = \sum_{k\geq 2} f_k$ with $f_k = S^{k-2} m S_k f$. For each $n$, the support condition \eqref{FsupportPi3} implies
$$S_n f_k \neq 0 \qquad \text{at most for }\;k= n-1,..., n+3.$$
For $N,K,L\in \N$ with $L\leq K<N-3$ the support condition and the $\mathcal R$-boundedness of $(S_n)_{n\geq 0}$ in $\mathscr L(L^p(\R^d,w;Y))$ as shown in Lemma \ref{Sk-Rbounded} yield
\begin{align*}
 \Big \|\sum_{n=0}^N r_n 2^{sn} S_n \sum_{k=L}^Kf_k \Big\|_{L^p(\Omega; L^p(\R^d,w;Y))}
 &\,\leq \sum_{j=-1}^3 \Big \|\sum_{n=L}^K r_n 2^{sn} S_n f_{n+j} \Big\|_{L^p(\Omega; L^p(\R^d,w;Y))}\\
 &\, \leq C\sum_{j=-1}^3 \Big \|\sum_{n=L}^K r_n 2^{sn}  f_{n+j} \Big\|_{L^p(\Omega; L^p(\R^d,w;Y))}.
 \end{align*}
Fix $j\in \{-1,...,3\}$. Then by Fubini's theorem, Lemma \ref{Slm-Rbounded} and Proposition \ref{prop:UMDHisF},
\begin{align*}
 \Big\|\sum_{n=L}^K r_n 2^{sn} &\,  f_{n+j} \Big\|_{L^p(\Omega; L^p(\R^d,w;Y))}^p \\
 &\, = \int_{\R^d}  \Big\|\sum_{n=L}^K r_n 2^{sn}  S^{n-2+j} m(x) S_{n+j} f(x) \Big\|_{L^p(\Omega;Y)}^p w(x)\,dx\\
 &\, \leq \mathcal R(\{S^l m(x):x\in \R^d, l\in \N\})^p \int_{\R^d} \Big \|\sum_{n=L}^K r_n 2^{sn}   S_{n+j} f(x)\Big \|_{L^p(\Omega;X)}^p w(x)\,dx\\
 &\, \leq (C\mathcal R(m))^p\Big \|\sum_{n=L}^\infty r_n 2^{sn}   S_{n+j} f \Big\|_{L^p(\Omega; L^p(\R^d,w;X))}^p=:(C\mathcal R(m))^p A_L^p.
\end{align*}
Here $\sum_{n=L}^\infty r_n 2^{sn}   S_{n+j} f$ converges in $L^p(\Omega; L^p(\R^d,w;X))$ by Proposition \ref{prop:UMDHisF}, and thus $A_L \to 0$ as $L\to \infty$. It follows from Proposition \ref{prop:UMDHisF} that $\sum_{k=L}^K f_k \in H^{s,p}(\R^d,w;Y)$ and
\[\Big\|\sum_{k=L}^K f_k\Big\|_{H^{s,p}(\R^d,w;Y)}\leq  C\sup_{N\geq 0}\Big \|\sum_{n=0}^N r_n 2^{sn} S_n \sum_{k=L}^K f_k \Big \|_{L^p(\Omega; L^p(\R^d,w;Y))}\leq C \mathcal R(m) A_L.\]
We conclude that $\big(\sum_{k=0}^N f_k\big)_{N\geq 0}$ is a Cauchy sequence in $H^{s,p}(\R^d,w;Y)$. Hence $\Pi_1(m,f) = \sum_{k=0}^\infty f_k$ converges in $H^{s,p}(\R^d,w;Y)$ and, again by Proposition \ref{prop:UMDHisF},
\begin{equation*}
\big \| \Pi_1(m,f)\big\|_{H^{s,p}(\R^d,w;Y)}\leq  C\mathcal R(m) A_0 \leq C \mathcal R(m) \|f\|_{H^{s,p}(\R^d,w;X)}.\qedhere
\end{equation*}
\end{proof}

The corresponding estimate of $\Pi_1$ for $F$-spaces is more elementary and does not need the UMD property of the underlying Banach spaces. Here and in the sequel, for $m:\R^d\to \calL(X,Y)$ we write
$$\|m\|_\infty = \sup_{x\in \R^d} \|m(x)\|_{\calL(X,Y)}.$$
We will make use of a convergence criterion from Lemma \ref{para2} in the appendix.

\begin{lemma} \label{multiplication2} Let $X$ and $Y$ be Banach spaces, $s\in \R$, $p\in (1,\infty)$, $q\in [1,\infty]$ and $w\in A_\infty$.  Let $m:\R^d \to \calL(X,Y)$ be strongly measurable and assume that the image of $m$ is bounded. Then for all $f\in F^{s}_{p,q}(\R^d,w;X)$ the limit $\Pi_1(m,f)$ exists in $\TD(\R^d;Y)$ and
$$\|\Pi_1(m,f)\|_{F^{s}_{p,q}(\R^d,w;Y)} \leq C \|m\|_{\infty}\|f\|_{F^{s}_{p,q}(\R^d,w;X)}.$$
\end{lemma}
\begin{proof}
Let again $\Pi_1(m,f) = \sum_{k\geq 2} f_k$ with $f_k = S^{k-2}m S_k f$.  We apply the estimate \eqref{6011} of Lemma \ref{para2}. It follows from \eqref{FsupportPi3} that the support condition \eqref{6000} holds. Therefore, $q=1$ and $w\in A_{\infty}$ are included. To check that the corresponding right-hand side of \eqref{6011} is finite we estimate
\[
\big\| \big( 2^{sk} f_k\big)_{k\geq 2}\big\|_{L^p(\R^d,w; \ell^q(Y))} \leq C \sup_{k\geq 0}\|S^{k}m\|_{\infty} \|f\|_{F^{s}_{p,q}(\R^d,w;X)}.
\]
Using $S^km =  2^{kd} \varphi_0(2^k\cdot)*m$ and $\|2^{kd} \varphi_0(2^k\cdot)\|_{L^1(\R^d)} = \|\varphi_0\|_{L^1(\R^d)}$, Young's inequality implies
$$\|S^{k}m\|_{\infty} \leq \|\varphi_0\|_{L^1(\R^d)} \|m\|_{\infty}, \qquad k\geq 0.$$
Hence $\Pi_1(m,f)$ exists by Lemma \ref{para2} and the asserted estimate holds true.
\end{proof}

\subsection{Special estimates of $\Pi_2$ and $\Pi_3$} We now estimate $\Pi_2$ and $\Pi_3$ as it is needed for the multiplication with the characteristic function $\one_{\R_+^d}$ of the half-space. Here we specialize to power weights of the form
$$w_\gamma(x',t) = |t|^\gamma, \qquad \gamma\in (-1,p-1),$$
and consider functions $m$ which depend on the last coordinate $t$ only. Following the considerations of \cite{Franke86} and \cite[Section 4.6.2]{RS96}, the main tools are Jawerth-Franke embeddings and convergence criteria for weighted spaces of entire analytic functions, as presented in Appendix \ref{sec:analytic}. In the rest of this subsection we can allow for general Banach spaces $X$ and $Y$.

To explain the parameters below, recall from \cite[Proposition 9.1.5]{GraModern} that for $w_{\gamma}$ the dual weight with respect to $p\in (1,\infty)$  is given by $w_{\gamma'}$, where
\begin{equation}\label{6016}
\gamma' = - \frac{\gamma}{p-1}, \qquad \frac{1+\gamma'}{p'} = 1-\frac{1+\gamma}{p}.
\end{equation}

\begin{lemma}\label{multiplication3}Let $X$ and $Y$ be Banach spaces, $p\in (1,\infty)$, $\gamma \in (-1,p-1)$ and
$- \frac{1+\gamma'}{p'} < s < \frac{1+\gamma}{p}.$ Let the numbers $r$ and  $\mu$  satisfy
\begin{equation}\label{6013}
1<r< \infty,\qquad \mu = 0, \quad  \qquad \text{ in case } \; 0\leq s <\frac{1+\gamma}{p},
\end{equation}
\begin{equation}\label{6014}
1<r<  \frac{1}{-s},\qquad \mu = 0, \quad  \qquad \text{ in case }\; -\frac{1}{p'} < s < 0,
\end{equation}
\begin{equation}\label{6015}
1<r<p',\qquad  \frac{\mu}{r} = -s - \frac{1}{p'} +\varepsilon,\quad  \qquad \text{ in case }\; - \frac{1+\gamma'}{p'} < s \leq -\frac{1}{p'},
\end{equation}
for some $\varepsilon > 0$. Let $m \in B_{r,\infty}^{\frac{1+\mu}{r}}(\R, w_\mu;\calL(X,Y))$ and consider it as a distribution on $\R^d$ which only depends on the last coordinate. Then for all $f\in F_{p,\infty}^s(\R^d,w_{\gamma};X)$ the limit $\Pi_2(m,f)$ exists in $\TD(\R^d;Y)$ and
 $$\|\Pi_2(m,f)\|_{F_{p,1}^s(\R^d,w_{\gamma};Y)} \leq C \|m\|_{B_{r,\infty}^{\frac{1+\mu}{r}}(\R, w_\mu;\calL(X,Y))} \|f\|_{F_{p,\infty}^s(\R^d,w_{\gamma};X)}.$$
\end{lemma}

\begin{remark} \label{rem:improve} In the estimate, for the microscopic parameters we have $q =1$ on the left-hand side and $q = \infty$ on the right-hand side. Such a microscopic improvement is possible because only special frequencies of $mf$ are in $\Pi_2$. Combined with $F_{p,1}^s\hookrightarrow H^{s,p}\hookrightarrow F_{p,\infty}^s$, it immediately gives an estimate of $\Pi_2$ in the Bessel-potential spaces.
\end{remark}

\begin{proof}[Proof of Lemma \ref{multiplication3}]
For a clearer presentation we assume that $\sum_{k=0}^\infty  S_{k+j}m S_k f$ exist for $j\in \{-1,0,1\}$ in $\TD(\R^d;Y)$, such that then also $\Pi_2(m,f)$ exists. This will be justified by means of Lemma \ref{para2} and the estimates in Step 3. In Step 4 we will show how the numbers $p_1$, $p_2$, $\gamma_1$ and $\gamma_2$ introduced in the first two steps can be chosen.

Recall the mixed-norm spaces $L^{p(p_1)}(\R^d,w;X) = L^p(\R^{d-1}; L^{p_1}(\R,w_{\gamma_1};X))$.

\emph{Step 1.} Suppose $p_1$ and $\gamma_1$ satisfy
\begin{equation}\label{Pi2_cond_1}
  1<p_1<p, \qquad -1 < \gamma_1 < p_1-1, \qquad \frac{\gamma_1}{p_1}\geq \frac{\gamma}{p}, \qquad s - \frac{1+\gamma}{p} + \frac{1+\gamma_1}{p_1}   > 0.
\end{equation}
For each $n$ the Fourier support of $S_n\big( \sum_{k=0}^\infty  S_{k+j}m S_k f\big)$ is contained in $\{|\xi|\leq  3\cdot 2^{n}\}$. The Jawerth-Franke embedding \eqref{jf_discrete_BF} thus gives
\begin{align*}
\|\Pi_2(m,f)\|_{F_{p,1}^s(\R^d,w_{\gamma};Y)}&\,  \leq \sum_{j=-1}^1 \Big \| \Big ( 2^{sn} S_n\sum_{k=0}^\infty  S_{k+j}m S_k f\Big)_{n \geq 0} \Big \|_{L^p(\R^d,w_{\gamma}; \ell^1(Y))}\\
&\, \leq C \sum_{j=-1}^1 \Big \| \Big ( 2^{(s - \frac{1+\gamma}{p} + \frac{1+\gamma_1}{p_1})n} S_n \sum_{k=0}^\infty  S_{k+j}m S_k f\Big)_{n \geq 0} \Big \|_{\ell^p(L^{p(p_1)}(\R^d,w_{\gamma_1};Y))}.
\end{align*}
Fix $j\in \{-1,0,1\}$. Due to \eqref{FsupportPi2}, the Fourier supports of $(S_{k+j}m S_k f)_{k\geq 0}$ are subject to \eqref{6001}. Since $w_{\gamma_1}\in A_{p_1}$ and $s - \frac{1+\gamma}{p} + \frac{1+\gamma_1}{p_1} > 0$, we may apply \eqref{6010} with $q = p > 1$ to obtain
\begin{align}
 \Big \| \Big ( 2^{(s - \frac{1+\gamma}{p} + \frac{1+\gamma_1}{p_1})n} S_n \sum_{k=0}^\infty  &\, S_{k+j}m S_k f\Big)_{n \geq 0} \Big \|_{\ell^p(L^{p(p_1)}(\R^d,w_{\gamma_1};Y))}\nonumber\\
 &\,\leq C \Big \| \Big ( 2^{(s - \frac{1+\gamma}{p} + \frac{1+\gamma_1}{p_1})k}  S_{k+j}m S_k f\Big)_{k \geq 0} \Big \|_{\ell^p(L^{p(p_1)}(\R^d,w_{\gamma_1};Y))}.\label{701}
\end{align}

\emph{Step 2.} Suppose $p_2$ and $\gamma_2$ satisfy
\begin{equation}\label{Pi2_cond_2}
 p < p_2 < \infty, \qquad -1 < \gamma_2 < p_2-1, \qquad  \frac{\gamma}{p}\geq \frac{\gamma_2}{p_2}.
\end{equation}
Define the numbers $r$ and $\mu$ by
$$\frac1r = \frac{1}{p_1} - \frac{1}{p_2}, \qquad \frac{\mu}{r} = \frac{\gamma_1}{p_1} - \frac{\gamma_2}{p_2}.$$
It follows from H\"older's inequality, applied in the last coordinate $t$ with exponent $\frac{p_2}{p_1} > 1$, that
\begin{align}
\Big \| \Big ( 2^{(s - \frac{1+\gamma}{p} + \frac{1+\gamma_1}{p_1})k}&\,  S_{k+j}m S_k f\Big)_{k \geq 0} \Big \|_{\ell^p(L^{p(p_1)}(\R^d,w_{\gamma_1};Y))}\notag\\
&\, \leq   \Big \| \Big( \Big \| 2^{(\frac{1+\gamma_1}{p_1}- \frac{1+\gamma_2}{p_2})k} S_k m  \Big \|_{L^{r}(\R, w_\mu; \calL(X,Y))} \Big)_{k\geq 0} \Big \|_{\ell^\infty(L^\infty(\R^{d-1}))}\label{667} \\
&\, \qquad \qquad \times \Big \| \Big (   2^{(s- \frac{1+\gamma}{p} + \frac{1+\gamma_2}{p_2})n}    S_{n} f\Big)_{n\geq 0}\Big\|_{\ell^p(L^{p(p_2)}(\R^d,w_{\gamma_2};X))}.\notag
\end{align}
For the second factor we use the Jawerth-Franke embedding \eqref{jf_discrete_FB}, which gives
\begin{align*}
\Big \| \Big (   2^{(s- \frac{1+\gamma}{p} + \frac{1+\gamma}{p_2})n}    S_{n} f\Big)_{n\geq 0} \Big\|_{\ell^p(L^{p(p_2)}(\R^d,w_{\gamma_2};X))} &\leq C  \Big \|(2^{sn}  S_nf)_{n \geq 0}\Big \|_{L^p(\R^d,w_{\gamma};\ell^\infty(X))} \\ & = C \|f\|_{F_{p,\infty}^s(\R^d,w_{\gamma};X)}.
\end{align*}
Consider the first factor. Since $m$ does not depend on $x'\in \R^{d-1}$, it is elementary to see that
$$S_k m = \mathcal F^{-1} ( \wh{\varphi}_k \wh{m}) = \mathcal F_t^{-1} (  \wh{\varphi}_k(0,\cdot) \mathcal F_t m).$$
Observe further that $(\mathcal F_t^{-1} \wh{\varphi}_k(0,\cdot))_{k\geq 0} \in \Phi(\R)$. Therefore
\begin{align*}
 \Big \|  \Big (\Big \|2^{(\frac{1+\gamma_1}{p_1}- \frac{1+\gamma_2}{p_2})k} &\,S_k m \Big\|_{L^{r}(\R, w_\mu; \calL(X,Y))} \Big)_{k\geq 0}  \Big \|_{\ell^\infty(L^\infty(\R^{d-1}))}= \|m\|_{B_{r,\infty}^{\sigma}(\R,w_\mu;\calL(X,Y))},
\end{align*}
where we have set $$\sigma = \frac{1+\mu}{r} = \frac{1+\gamma_1}{p_1}- \frac{1+\gamma_2}{p_2}.$$

\emph{Step 3.}  In the next step we find $p_1$, $\gamma_1$, $p_2$ and $\gamma_2$ satisfying \eqref{Pi2_cond_1} and \eqref{Pi2_cond_2}. Then it follows from \eqref{667} that
$$\big(2^{(s- \frac{1+\gamma}{p} + \frac{1+\gamma_1}{p_1})k} S_{k+j}m S_kf \big)_{k\geq 0} \in \ell^p(L^{p(p_1)}(\R^d,w_{\gamma_1};Y)).$$
Thus $\sum_{k\geq 0} S_{k+j}m S_k$ exists in $\TD(\R^d,Y)$ for $j\in \{-1,0,1\}$ by Lemma \ref{para2} and the estimate \eqref{701} is valid. Hence also $\Pi_2(m,f)$ exists, and the considerations of Step 1 show that it can be estimated as asserted.

\emph{Step 4.} Here and in the sequel, by $a\searrow b$ we mean that $a$ is chosen larger but arbitrarily close to $b$. Similar for $a\nearrow b$. We seek for parameters $p_1$, $\gamma_1$, $p_2$, $\gamma_2$ satisfying \eqref{Pi2_cond_1} and \eqref{Pi2_cond_2} such that $B_{r,\infty}^{\sigma}(\R,w_\mu;\calL(X,Y))$ becomes as large as possible. For each admissible choice of these parameters we have $\sigma = \frac{1+\mu}{r}$ and $\frac{\mu}{r} = \frac{\gamma_1}{p_1} - \frac{\gamma_2}{p_2}\geq 0.$ In view of the necessary and sufficient conditions for Sobolev embeddings from \cite[Theorem 1.1]{MeyVer1}, we thus aim to minimize $\sigma$ and $\frac{\mu}{r}$. In any case, the choices
$$p_2 \searrow p, \qquad \gamma_2 = \frac{\gamma}{p} p_2,$$
are optimal in this sense and satisfy \eqref{Pi2_cond_2}.

\emph{Substep 4.1.} Let $0\leq s < \frac{1+\gamma}{p}$ as in \eqref{6013}. Here the choices
$$p_1 \nearrow p, \qquad \gamma_1 =  \frac{\gamma}{p} p_1,$$ satisfy \eqref{Pi2_cond_1}. This leads to $\mu = 0$ and that $r$ may be arbitrarily large.

\emph{Substep 4.2.} Let $\frac{1}{p}-1 < s< 0$ as in \eqref{6014}. Choosing $ \gamma_1 =  \frac{\gamma}{p} p_1$, to satisfy $s -\frac{1+\gamma}{p} + \frac{1+\gamma_1}{p_1}>0$ we have to restrict to $1<p_1 < \frac{1}{\frac{1}{p}-s}$. It is possible to choose such $p_1$ by assumption in this substep. For $p_1 \nearrow \frac{1}{\frac{1}{p}-s}$ the condition \eqref{Pi2_cond_1} is indeed satisfied. This results  in $\mu = 0$ and $r < \frac{1}{-s}$.

\emph{Substep 4.3.}  Let  $\frac{1+\gamma}{p} - 1< s \leq \frac{1}{p}-1$ as in \eqref{6015}. This is only possible for $\gamma < 0$. Here $\frac{\gamma_1}{p_1} = \frac{\gamma}{p}$ is not allowed, since there is no $p_1 > 1$ with $s - \frac{1}{p} + \frac{1}{p_1} > 0$. So we choose
$$p_1\searrow 1, \qquad \gamma_1 \searrow \Big(\frac{1+\gamma}{p} -s\Big)p_1 - 1.$$
First this gives $r<p'$. Write $p_1 = 1+\varepsilon_1$ and $\frac{\gamma_1}{p_1} = \frac{1+\gamma}{p} -s - \frac{1}{p_1} + \varepsilon_2$, where $\varepsilon_1,\varepsilon_2>0$. Then $\frac{\mu}{r} = \frac{\gamma_1}{p_1} - \frac{\gamma_2}{p_2} = -s -\frac{1}{p'} + \varepsilon_1+ \varepsilon_2$. Setting $\varepsilon = \varepsilon_1 + \varepsilon_2$, we may thus choose $r$ and $\mu$ as asserted.
 \end{proof}

The estimate of $\Pi_3$ is similar. Again there is a microscopic improvement, see Remark \ref{rem:improve}.

\begin{lemma}  \label{multiplication4} Let $X$ and $Y$ be Banach spaces, $p\in (1,\infty)$, $\gamma \in (-1,p-1)$ and $- \frac{1+\gamma'}{p'} < s < \frac{1+\gamma}{p}.$ Let the numbers $r $ and $\mu$ satisfy
\begin{equation}\label{6018}
 1<r< \infty,\qquad \frac{\mu}{r} = s - \frac{1}{p}  +\varepsilon,  \quad\qquad  \text{ in case }\; \frac{1}{p} \leq s < \frac{1+\gamma}{p},
 \end{equation}
\begin{equation}\label{6019}
1<r<  \frac{1}{s},\qquad \mu = 0,\quad \qquad \text{ in case }\; 0< s <\frac{1}{p},
\end{equation}
\begin{equation}\label{6020}
1<r<\infty,\qquad  \mu =0, \quad \qquad \text{ in case }\; - \frac{1+\gamma'}{p'} < s \leq 0,
\end{equation}
for some $\varepsilon > 0$. Let  $m \in B_{r,\infty}^{\frac{1+\mu}{r}}(\R, w_\mu;\calL(X,Y))$ and consider it as a distribution on $\R^d$ which only depends on the last coordinate. Then for all $f\in F_{p,\infty}^s(\R^d,w_{\gamma};X)$ the limit $\Pi_3(m,f)$ exists in $\TD(\R^d;Y)$ and
 $$\|\Pi_3(m,f)\|_{F_{p,1}^s(\R^d,w_{\gamma};Y)} \leq C \|m\|_{B_{r,\infty}^{\frac{1+\mu}{r}}(\R, w_\mu;\calL(X,Y))} \|f\|_{F_{p,\infty}^s(\R^d,w_{\gamma};X)}.$$
\end{lemma}

\begin{proof}
\emph{Step 1.} As in the previous lemma we assume that $\Pi_3(m,f)$ exists in $\TD(\R^d;Y)$ from the beginning and justify this afterwards by means of Lemma \ref{para2}.

Let $p_1$ and $\gamma_1$ be such that
\begin{equation}\label{Pi3_cond_1}
 1<p_1 < p, \qquad -1 < \gamma_1 < p_1-1, \qquad  \frac{\gamma_1}{p_1}\geq \frac{\gamma}{p}.
\end{equation}
Since the Fourier supports of the summands of $\Pi_3(m,f)$ satisfy \eqref{FsupportPi3}, we may use  \eqref{6011} under the assumption \eqref{6000}, where we can allow for $q= 1$, and then the Jawerth-Franke embedding \eqref{jf_discrete_BF} to obtain
\begin{align*}
\|\Pi_3(m,f) \|_{F_{p,1}^s(\R^d,w_{\gamma};Y)}
&\, \leq C \Big \| \big ( 2^{sk}  S_{k}m S^{k-2} f\big)_{k \geq 2} \Big \|_{L^p(\R^d,w_{\gamma}; \ell^1(Y))}\\
&\, \leq C \Big \| \Big ( 2^{(s- \frac{1+\gamma}{p} +\frac{1+\gamma_1}{p_1})k}  S_{k}m S^{k-2} f \Big)_{k \geq 0} \Big\|_{\ell^p(L^{p(p_1)}(\R^d,w_{\gamma_1};Y))}.
\end{align*}
Now let $p_2$ and $\gamma_2$ satisfy
\begin{equation}\label{Pi3_cond_2}
 p<p_2 < \infty, \qquad -1 < \gamma_2 < p_2-1, \qquad  \frac{\gamma}{p}\geq \frac{\gamma_2}{p_2},\qquad s + \frac{1+\gamma_2}{p_2} - \frac{1+\gamma}{p} < 0,
\end{equation}
and set $w_{\gamma_2}(x',t) = |t|^{\gamma_2}$. Then, by H\"older's inequality,
\begin{align*}
\Big \| \Big ( 2^{(s - \frac{1+\gamma}{p} +\frac{1+\gamma_1}{p_1})k}&\,   S_{k}m S^{k-2}f \Big)_{k \geq 0} \big\|_{\ell^p(L^{p(p_1)}(\R^d,w_{\gamma_1};Y))}\\
&\, \leq   \Big \| \big ( \big \|  2^{\sigma k} S_k m  \big \|_{L^{r}(\R, w_\mu; \calL(X,Y))}\big)_{k\geq 0} \Big \|_{\ell^\infty(L^\infty(\R^{d-1}))} \\
&\, \qquad \qquad \times \Big \| \Big (   2^{(s- \frac{1+\gamma}{p} + \frac{1+\gamma_2}{p_2})n}    S^{n} f\Big)_{n\geq 0}\Big\|_{\ell^p(L^{p(p_2)}(\R^d,w_{\gamma_2};X))},
\end{align*}
where as before $r = \frac{1}{p_1} - \frac{1}{p_2}$, $\mu = (\frac{\gamma_1}{p_1} - \frac{\gamma_2}{p_2}) r$ and $\sigma = \frac{1+\gamma_1}{p_1} - \frac{1+\gamma_2}{p_2}$. Since $s + \frac{1+\gamma_2}{p_2} - \frac{1+\gamma}{p} < 0$ we can apply Lemma \ref{lem:para1} to replace $S^{n}$ by $S_n$ in the second factor, which can then be estimated by $C\|f\|_{F_{p,\infty}^s(\R^d,w;X)}$ in the same way as in the previous lemma using \eqref{jf_discrete_FB}. Also the first factor can be treated in the same way to obtain
$$\Big \| \big ( \big \|  2^{\sigma k} S_k m  \big \|_{L^{r}(\R, w_\mu; \calL(X,Y))}\big)_{k\geq 0} \Big \|_{\ell^\infty(L^\infty(\R^{d-1}))}  = \|m\|_{B_{r,\infty}^{\sigma}(\R,w_\mu;\calL(X,Y))}.$$

\emph{Step 2.} We enlarge $B_{r,\infty}^{\sigma}(\R,w_\mu;\calL(X,Y))$ by choosing optimal parameters according to \eqref{Pi3_cond_1} and \eqref{Pi3_cond_2}. In any case $p_1\nearrow p$ and $\gamma_1 = \frac{\gamma}{p}p_1\in (1,p_1-1)$ satisfies \eqref{Pi3_cond_1} and is the best choice.

\emph{Substep 2.1}  Let $\frac{1+\gamma}{p} - 1 < s \leq 0$ as in \eqref{6020}. Then $p_2\searrow p$ and $\gamma_2 = \frac{\gamma}{p} p_2$ are admissible, which leads to  $\mu = 0$ and that $r$ may be arbitrarily large.

\emph{Substep 2.2}  Let $0< s < \frac{1}{p}$ as in \eqref{6019}. We still take $\gamma_2 = \frac{\gamma}{p} p_2$, but then we have to restrict to $p_2 \nearrow \frac{1}{\frac{1}{p} -s}$. This gives  $\mu = 0$ and $r\nearrow \frac{1}{s}$.

\emph{Substep 2.3}  Let $\frac{1}{p}\leq s < \frac{1+\gamma}{p}$ as in \eqref{6018}. Then we cannot take $\gamma_2 = \frac{\gamma}{p} p_2$, since $s - \frac{1+\gamma}{p} + \frac{1+\gamma_2}{p_2} < 0$ becomes impossible. Instead we let $p_2\nearrow \infty$ and $\gamma_2 \searrow (\frac{1+\gamma}{p} -s)p_2 -1$, which satisfies \eqref{Pi3_cond_2}. Writing $\frac{\gamma_2}{p_2} = \frac{1+\gamma}{p} -s -\varepsilon$ with $\varepsilon>0$, we get $\frac{\mu}{r} = s-\frac{1}{p} + \varepsilon$ and $r\nearrow p$.
\end{proof}

\section{Pointwise multiplication} \label{sec:p-mult}

\subsection{Irregular functions}\label{subsec:irr}

In this section we combine the estimates for the paraproducts to obtain sufficient conditions for the boundedness of
$f\mapsto mf$ for irregular $m$ and vector-valued functions $f$ in Besov spaces, Triebel-Lizorkin spaces and Bessel-potential spaces. The result extends \cite[Theorem 3.4.2]{Franke86} to the weighted vector-valued setting, see also \cite[Corollary 4.6.2/1]{RS96} and \cite[Section 2.8]{Tri83}. In these works also the cases $p,q\leq 1$ are considered.

Recall that the product of distributions is given by $mf = \lim_{l\to\infty} S^l m \cdot S^l f$ (if the limit exists).

\begin{theorem} \label{multiplication-esti}
Let $X$ and $Y$ be  Banach spaces, $p\in (1,\infty)$, $q\in[1,\infty]$,  $\gamma \in (-1,p-1)$ and $- \frac{1+\gamma'}{p'} < s < \frac{1+\gamma}{p}.$ Let the numbers $r$ and $\mu$  satisfy
$$ 1<r< \frac{1}{|s|},\qquad \mu =0,  \qquad \text{ in case }\;- \frac{1}{p'} < s < \frac{1}{p},$$
$$1 < r < p, \qquad \frac{\mu}{r}  = s- \frac{1}{p} + \varepsilon,   \qquad \text{ in case }\;\frac{1}{p}\leq s < \frac{1+\gamma}{p},$$
$$1 < r < p', \qquad \frac{\mu}{r}  = - s- \frac{1}{p'} + \varepsilon, \qquad  \text{ in case }\;-\frac{1+\gamma'}{p'} < s \leq -\frac{1}{p'},$$
for some $\varepsilon > 0$. Let $m \in B_{r,\infty}^{\frac{1+\mu}{r}}(\R, w_\mu; \calL(X,Y))\cap L^\infty(\R;\calL(X,Y))$ and consider it as a distribution on $\R^d$ which only depends on the last coordinate. Then the following holds true.
\begin{itemize}
 \item[\text{(a)}] For $\mathcal A\in \{F,B\}$ and $f\in \mathcal A_{p,q}^s(\R^d,w_{\gamma};X)$ the product $mf$ exists in $\TD(\R^d;Y)$ and \begin{equation*}
\|mf\|_{\mathcal A_{p,q}^s(\R^d,w_{\gamma};Y)} \leq C \big(\|m\|_{B_{r,\infty}^{\frac{1+\mu}{r}}(\R, w_\mu; \calL(X,Y))} + \|m\|_{\infty}  \big)\|f\|_{\mathcal A_{p,q}^s(\R^d,w_{\gamma};X)}.
\end{equation*}
 \item[\text{(b)}] If $X$ and $Y$ have UMD and if the image of $m$ is $\mathcal R$-bounded by $\mathcal R(m)$, then for all $f\in H^{s,p}(\R^d,w_{\gamma};X)$ the product $mf$ exists in $\TD(\R^d;Y)$ and
 \begin{equation*}
\|mf\|_{H^{s,p}(\R^d,w_{\gamma};Y)} \leq C \big( \|m\|_{B_{r,\infty}^{\frac{1+\mu}{r}}(\R, w_\mu; \calL(X,Y))} + \mathcal R(m)\big)\|f\|_{H^{s,p}(\R^d,w_{\gamma};X)}.
\end{equation*}
\end{itemize}
\end{theorem}

\begin{proof}
Consider Assertion (a) for $\mathcal A = F$. The paraproducts exist in $\TD(\R^d;Y)$ by the Lemmas \ref{multiplication2}, \ref{multiplication3} and \ref{multiplication4}, hence $mf$ exists as a distribution. The estimate follows from the monotonicity of the $F$-spaces with respect to $q\in [1,\infty]$.  For  $\mathcal A = B$,  the estimate is now a consequence of real interpolation, see \cite[Proposition 6.1]{MeyVer1}. Assertion (b) follows from the Lemmas  \ref{multiplication1}, \ref{multiplication3} and \ref{multiplication4}, combined with the elementary embeddings \eqref{eq:TriebelLizorkinH}.
\end{proof}

\begin{remark} \label{R-bound-crit}\
\begin{enumerate}[(i)]
\item If $m$ is scalar-valued, then $\mathcal R(m) \leq 2 \|m\|_{L^\infty(\R^d)}$ by \cite[Proposition 2.5]{KuWe}. However, also in this case our methods do not allow to remove the UMD property of the underlying Banach spaces. If $X,Y$ are Hilbert spaces, then a family of linear operators is $\mathcal R$-bounded if and only if it is bounded, see e.g. \cite[Remark 3.2]{DHP}.
\item A scaling argument shows that an  estimate as above is only possible with a space $B_{r,\infty}^\sigma(\R,w_\mu)$ satisfying $\sigma = \frac{1+\mu}{r}$. The conditions on the parameters cannot be improved by duality arguments in case of reflexive spaces.
\item Since $ B_{r,\infty}^{\frac{1}{r}}$ is not embedded into $L^\infty$, in the sharp case the boundedness of $m$ does not follow from the Besov regularity and must be imposed as an extra condition. Sufficient conditions for the  $\mathcal R$-boundedness of the image of $m$ in terms of Besov regularity are provided in \cite[Theorem 5.1]{HytVer}. In particular, if $m\in B_{r,1}^{1/r}(\R; \mathscr L(X,Y))$ for some $r$ which depends on type and cotype of $X$ and $Y$ (see Section \ref{sec:type}), then the image of $m$ is automatically $\mathcal R$-bounded.

\end{enumerate}
\end{remark}

Analogous arguments yield multiplication estimates for radial power weights of the form
$$v_\gamma(x) = |x|^\gamma.$$

\begin{theorem} \label{multiplication-esti-d} Let $X$ and $Y$ be  Banach spaces, $p\in (1,\infty)$, $q\in[1,\infty]$, $\gamma\in (-d,d(p-1))$ and $- \frac{d+\gamma'}{p'} < s < \frac{d+\gamma}{p}$. Let the numbers $r$ and $\mu$ satisfy
$$ 1<r< \frac{d}{|s|},\qquad \mu =0, \quad  \qquad \text{ in case}\;- \frac{d}{p'} < s < \frac{d}{p},$$
$$1 < r < p, \qquad \frac{\mu}{r}  = s- \frac{d}{p} + \varepsilon, \qquad \text{ in case }\;\frac{d}{p}\leq s < \frac{d+\gamma}{p},$$
$$1 < r < p', \qquad \frac{\mu}{r}  = - s- \frac{d}{p'} + \varepsilon,   \qquad \text{ in case }\;- \frac{d+\gamma'}{p'} < s \leq -\frac{d}{p'},$$
for some $\varepsilon > 0$. Suppose that $m \in B_{r,\infty}^{\frac{d+\mu}{r}}(\R^d, v_\mu; \calL(X,Y))\cap L^\infty(\R^d;\calL(X,Y))$. Then for $\mathcal A\in \{F,B\}$ and $f\in \mathcal A_{p,q}^s(\R^d,v_{\gamma};X)$ the product $mf$ exists in $\TD(\R^d;Y)$ and
$$
\|mf\|_{\mathcal A_{p,q}^s(\R^d,v_\gamma;Y)} \leq C \big(\|m\|_{B_{r,\infty}^{\frac{d+\mu}{r}}(\R^d, v_\mu; \calL(X,Y))} + \|m\|_{\infty} \big)\|f\|_{\mathcal A_{p,q}^s(\R^d,v_\gamma;X)}.
$$
Moreover, if $X$ and $Y$ have UMD and if the image of $m$ is $\mathcal R$-bounded by $\mathcal R(m)$, then for all $f\in H^{s,p}(\R^d,v_\gamma;X)$ the product $mf$ exists in $\TD(\R^d;Y)$ and
$$
\|mf\|_{H^{s,p}(\R^d,v_\gamma;Y)} \leq C \big(\|m\|_{B_{r,\infty}^{\frac{d+\mu}{r}}(\R^d, v_\mu; \calL(X,Y))} + \mathcal R(m)  \big)\|f\|_{H^{s,p}(\R^d,v_\gamma;X)}.
$$
\end{theorem}
\begin{proof} As before one decomposes $mf$ into the paraproducts. For the estimate of $\Pi_1(m,f)$ one can apply the Lemmas \ref{multiplication1} and \ref{multiplication2}. To estimate $\Pi_2(m,f)$ one argues as in Lemma \ref{multiplication3}. Instead of \eqref{jf_discrete_BF} and \eqref{jf_discrete_FB} one directly uses the Jawerth-Franke embeddings from \cite[Theorem 6.4]{MeyVer1} for radial weights. One further uses \eqref{6011} instead of \eqref{6005}. The optimal choice of the parameters is analogous. In a similar way one modifies the proof of Lemma \ref{multiplication4} to estimate $\Pi_3(m,f)$.
\end{proof}

\subsection{H\"older continuous functions}
In this section we investigate the boundedness of $f\mapsto m f$ for smooth and bounded functions $m$ on $B$- and $F$-spaces. The case of $H$-spaces was already considered in Proposition \ref{thm:mult-smooth2}. Together with Theorem \ref{multiplication-esti} this provides the right ingredients to prove the Theorems \ref{thm:1} and \ref{thm:2} later on.

\begin{proposition} \label{prop:mult-smooth}
Let $X$ and $Y$ be Banach spaces, $s\in \R$, $p\in (1,\infty)$, $q\in [1,\infty]$ and $w\in A_p$.
Assume that $m\in BC^\sigma(\R^d; \calL(X,Y))$ for some $\sigma>|s|$.
Then for $\mathcal A\in \{F, B\}$ and all $f\in \mathcal A_{p,q}^s(\R^d,w;X)$ the product $mf$ exists in $\TD(\R^d;Y)$ and
$$\|mf\|_{\mathcal A_{p,q}^s(\R^d,w;Y)} \leq C \|m\|_{BC^\sigma(\R^d; \calL(X,Y))}\|f\|_{\mathcal A_{p,q}^s(\R^d,w;X)}.$$
\end{proposition}

\begin{proof} By Lemma \ref{multiplication2} one has
$$\|\Pi_1(m,f)\|_{F_{p,q}^s(\R^d,w;Y)} \leq C \|m\|_{L^\infty(\R^d;\calL(X,Y))}\|f\|_{F_{p,q}^s(\R^d,w;X)}.$$
The $B$-case follows from real interpolation (see \cite[Proposition 5.1]{MeyVer1}). To estimate $\Pi_2(m,f)$ we use  $B_{p,\infty}^{s+\sigma}\hookrightarrow \mathcal A_{p,q}^s$ and that $s+\sigma >0$ to apply \eqref{6012} under the assumption \eqref{6001}, which gives
\begin{align*}
\|\Pi_2(m,f)\|_{\mathcal A_{p,q}^s(\R^d,w;Y)} &\, \leq C \|\Pi_2(m,f)\|_{B_{p,\infty}^{s+\sigma}(\R^d,w;Y)} \\
&\, \leq C \sum_{j=-1}^1\big\|\big(2^{(s+\sigma)k} S_{k+j}m S_k f)_{k\geq 0}\big\|_{\ell^\infty(L^p(\R^d,w;Y))}.
\end{align*}
Then for fixed $j$ we obtain
\begin{align*}
\big\|\big(2^{(s+\sigma)k} S_{k+j}&\,m S_k f)_{k\geq 0}\big\|_{\ell^\infty(L^p(\R^d,w;Y))}\\
&\, \leq C \big\|\big(2^{\sigma k} S_{k+j}m)_{k\geq 0}\big\|_{\ell^\infty(L^\infty(\R^d;\calL(X,Y)))} \big\|\big(2^{sk} S_k f)_{k\geq 0}\big\|_{\ell^\infty(L^p(\R^d,w;X))}\\
&\, \leq C \|m\|_{B_{\infty,\infty}^{\sigma}(\R^d; \calL(X,Y))}\|f\|_{B_{p,\infty}^s(\R^d,w;X)}\\
&\, \leq C \|m\|_{BC^{\sigma}(\R^d; \calL(X,Y))}\|f\|_{\mathcal A_{p,q}^s(\R^d,w;X)}.
\end{align*}
In the last line we have used that $BC^{\sigma} \hookrightarrow B_{\infty,\infty}^{\sigma}$, see \cite[Proposition 2.5.7]{Tri83} for the scalar case, and $\mathcal A_{p,q}^s\hookrightarrow B_{p,\infty}^s$. For $\Pi_3(m,f)$ we use $B_{p,1}^s \hookrightarrow  \mathcal A_{p,q}^s$ and apply \eqref{6012} under the assumption \eqref{6000} to get
\begin{align*}
\|\Pi_3(m,f)\|_{\mathcal A_{p,q}^s(\R^d,w;Y)} &\, \leq C \|\Pi_3(m,f)\|_{B_{p,1}^s(\R^d,w;Y)} \leq C \big\|\big(2^{sk} S_{k}m S^{k-2}f)_{k\geq 0}\big\|_{\ell^1(L^p(\R^d,w;Y))}\\
&\, \leq C \big\|\big(2^{\sigma k} S_km)_{k\geq 0}\big\|_{\ell^\infty(L^\infty(\R^d;\calL(X,Y)))} \big\|\big(2^{(s-\sigma)k} S^{k} f)_{k\geq 0}\big\|_{\ell^1(L^p(\R^d,w;X))}\\
&\, \leq C  \|m\|_{BC^{\sigma}(\R^d; \calL(X,Y))} \|\big(2^{(s-\sigma)k} S_{k} f)_{k\geq 0}\|_{\ell^1(L^p(\R^d,w;X))}\\
&\, \leq C \|m\|_{BC^{\sigma}(\R^d; \calL(X,Y))} \|f\|_{\mathcal A_{p,q}^s(\R^d,w;X)}.
\end{align*}
Here we also employed that $s-\sigma <0$ and applied Lemma \ref{lem:para1} to replace $S^k$ by $S_k$ in the second to last line. The existence of the paraproducts and thus of $mf$ is a consequence of these estimates and Lemma \ref{para2}.
\end{proof}

\subsection{Characteristic functions\label{subs:multich}}
It is well-known that the precise local regularity of $\one_{\R_+^d}$ is $B_{r,\infty}^{\frac{1}{r}}$, see  \cite[Lemma 4.6.3/2]{RS96} and the references therein. This information is actually not sufficient to apply Theorem \ref{multiplication-esti} in case $s\notin (- \frac{1}{p'}, \frac{1}{p})$.

\begin{lemma} \label{chi} For all $\phi\in C_c^\infty(\R^d)$, $p\in (1,\infty)$ and $\gamma \in (-1, p-1)$ one has
$$\one_{\R_+^d} \phi  \in B_{p,\infty}^{\frac{1+\gamma}{p}}(\R^d, w_{\gamma})\cap L^\infty(\R^d).$$
\end{lemma}
\begin{proof}
Let $Q = Q'\times (a,b)\subset \R^d$ be a cube with $\supp \phi \subset Q$. We write  $g_h = g(\cdot + h)$ for a translation by $h\in \R^d$. Clearly, $\one_{\R_+^d} \phi\in L^p(\R^d,w_\gamma)\cap L^\infty(\R^d)$.
By \eqref{remark-besov-norm},  for $\one_{\R_+^d}\phi \in B_{p,\infty}^{\frac{1+\gamma}{p}}(\R^d, w_{\gamma})$ it is sufficient to show that
$$[\one_{\R_+^d} \phi]_{B_{p,\infty}^{\frac{1+\gamma}{p}}(\R^d, w_{\gamma})} = \sup_{r>0} r^{-\frac{1+\gamma}{p}} \sup_{|h|\leq r} \|\one_{\R_+^d,h}\phi_h - \one_{\R_+^d}\phi\|_{L^p(\R^d,w_{\gamma})} < \infty.$$

\emph{Step 1.} Let $r\leq 1$ and $h\in \R^d$ with $|h|\leq r$. Then
\begin{align*}
 \|\one_{\R_+^d,h}\phi_h  - \one_{\R_+^d} \phi \|_{L^p(\R^d,w_{\gamma})} \leq \|\one_{\R_+^d,h} (\phi_h -\phi) \|_{L^p(\R^d,w_{\gamma})} + \|(\one_{\R_+^d,h} -\one_{\R_+^d})\phi\|_{L^p(\R^d,w_\gamma)}.
\end{align*}
For the first summand we estimate
\begin{align*}
 \|\one_{\R_+^d,h}  (\phi_h -\phi)\|_{L^p(\R^d,w_{\gamma})}^p &\,= \int_{\R^d} \one_{\R_+^d}(t+h_d) |\phi(x+h) - \phi(x)|^p  |t|^\gamma \,dt \,d x'\\
&\, \leq \int_{B(Q,1)}\|\phi'\|_\infty^p |h|^p |t|^\gamma \,dt \,d x' \leq C\, |h|^p,
\end{align*}
where $B(Q,1) = \{x\in \R^d: \text{dist}(x,Q)\leq 1\}$.
For the second summand we have
\begin{align*}
 \|(\one_{\R_+^d,h} -\one_{\R_+^d})\phi\|_{L^p(\R^d,w_{\gamma})}^p &\, = \int_{\R^d} |\one_{\R_+^d}(t+h_d) -\one_{\R_+^d}(t)| |\phi(x)|^p|t|^\gamma \,d t\,d x' \\
 &\, \leq C\int_{Q)\cap \{|t|\leq |h_d|\}}|t|^\gamma\,d x'\,d t \leq  C \int_{-h_d}^{h_d} |t|^\gamma\, dt \leq C |h|^{1+\gamma}.
\end{align*}
Therefore
$$\sup_{r\leq 1} r^{-\frac{1+\gamma}{p}} \sup_{|h|\leq r} \|\one_{\R_+^d,h}\phi_h - \one_{\R_+^d}\phi\|_{L^p(\R^d,w_{\gamma})} \leq C \sup_{r\in (0,1)} r^{-\frac{1+\gamma}{p}} (r + r^{\frac{1+\gamma}{p}}) < \infty.$$

\emph{Step 2.} Let $r\geq 1$ and $h\in \R^d$ with $|h|\leq r$. We have
\begin{align*}
\|\one_{\R_+^d,h}\phi_h - \one_{\R_+^d}\phi\|_{L^p(\R^d,w_{\gamma})}&\, \leq \|\one_{\R_+^d,h}\phi_h\|_{L^p(\R^d,w_{\gamma})} + \| \one_{\R_+^d}\phi\|_{L^p(\R^d,w_{\gamma})}.
\end{align*}
The second summand is independent of $h$. For the first summand we estimate
\begin{align*}
\|\one_{\R_+^d,h}\phi_h\|_{L^p(\R^d,w_{\gamma})}^p &\, \leq  \int_{Q'-h} \int_{a-h_d}^{b-h_d} |\phi(x+h)|^p |t|^\gamma\, dt\,dx'\\
&\,  \leq C  \int_{a-h_d}^{b-h_d} |t|^\gamma \,dt \leq C (1+|r|^\gamma).
\end{align*}
This yields
$$\sup_{r\geq 1} r^{-\frac{1+\gamma}{p}} \sup_{|h|\leq r} \|\one_{\R_+^d,h}\phi_h - \one_{\R_+^d}\phi\|_{L^p(\R^d,w_{\gamma})}<\infty. $$

Combining this with Step 1, it follows that $[\one_{\R_+^d} \phi]_{B_{p,\infty}^{\frac{1+\gamma}{p}}(\R^d, w_{\gamma})}$ is finite.\end{proof}

\begin{remark} For $\gamma \geq 0$ and $\phi$ nonvanishing around the origin we have $\one_{\R_+^d} \phi \in B_{p,q}^{\frac{1+\gamma}{p}}(\R^d,w_\gamma)$ if and only if $q= \infty$. In fact, $\one_{\R_+^d} \phi \in B_{p,q}^{\frac{1+\gamma}{p}}(\R^d,w_\gamma)$ implies that  $\one_{\R_+^d}\phi \in B_{p,q}^{\frac{1}{p}}(\R^d)$ by \cite[Theorem 1.1]{MeyVer1}, and the latter is true if and only if $q = \infty$ by \cite[Lemma 4.6.3/2]{RS96}. For $\gamma\in (-1,0)$ this argument does not work. Using the difference norm from Proposition \ref{prop:Lpsmoothness-Besov}, one can show that the characterization is true also for these powers.
\end{remark}

We can now prove our main results Theorems \ref{thm:1} and \ref{thm:2} on the multiplier property of $\one_{\R^d_+}$.
\begin{proof}[Proof of Theorems \ref{thm:1} and \ref{thm:2}]
Recall that $-\frac{1+\gamma'}{p'} = \frac{1+\gamma}{p} - 1$.
Let $\phi\in C_c^\infty(\R)$ be equal to $1$ for $|t|\leq 1$ and equal to zero for $|t|\geq 2$. Then $ \one_{\R_+^d} (1-\phi)$ belongs to $BC^\infty(\R^d)$ and is thus a pointwise multiplier by the Propositions \ref{thm:mult-smooth2} and \ref{prop:mult-smooth}. Considering $\one_{\R_+^d}  \phi$ to depend on the last coordinate $t$ only, Lemma \ref{chi} shows that
$$\one_{\R_+^d}  \phi\in B_{r,\infty}^{\frac{1+\mu}{r}}(\R,w_\mu) \cap L^\infty(\R)$$
for all $r\in (1,\infty)$ and all $\mu\in (-1,r-1)$. Now Theorem \ref{multiplication-esti} applies to $\one_{\R_+^d}  \phi$. Indeed, one can choose arbitrary $r\in (1, \frac{1}{|s|})$ for $s\in (-\frac{1}{p'},\frac{1}{p})$, $r$ close to $p$ for $s\in [\frac{1}{p}, \frac{1+\gamma}{p})$ and $r$ close to $p'$ for $s\in (-\frac{1+\gamma'}{p'},-\frac{1}{p'}]$. Since $\one_{\R_+^d}  \phi$ is scalar-valued, its image is $\mathcal R$-bounded, see Remark \ref{R-bound-crit}. \end{proof}

\subsection{Multiplication algebras, type and cotype} \label{sec:type}

The following classical result holds for $s>0$, $p,q\in [1,\infty]$ and $\mathcal{A}\in \{B,F\}$ (see \cite[Section 4.6.4]{RS96}),
\begin{equation}\label{eq:classicalcase1}
\|mf\|_{{\mathcal A}^{s}_{p,q}(\R^d)} \leq C \big(\|m\|_{L^\infty(\R^d)} \|f\|_{{\mathcal A}^{s}_{p,q}(\R^d)} + \|m\|_{{\mathcal A}^{s}_{p,q}(\R^d)}  \|f\|_{L^\infty(\R^d)}\big).
\end{equation}
In other words, ${\mathcal A}^{s}_{p,q}\cap L^\infty$ is a multiplicative algebra. Of course, if $s$ is large enough, then ${\mathcal A}^{s}_{p,q}\cap L^\infty = {\mathcal A}^{s}_{p,q}$ by Sobolev embedding.  Since in the scalar-valued case one has $H^{s,p} = F^{s}_{p,2}$ for $p\in (1,\infty)$, this includes an estimate for Bessel-potential spaces, i.e.,
\begin{equation}\label{eq:classicalcase2}
\|mf\|_{H^{s,p}(\R^d)} \leq C\big( \|m\|_{L^\infty(\R^d)} \|f\|_{H^{s,p}(\R^d)} + \|m\|_{H^{s,p}(\R^d)} \|f\|_{L^\infty(\R^d)}\big).
\end{equation}

Using the convergence criteria from Lemma \ref{para2}, the following extension of \eqref{eq:classicalcase1} to the weighted vector-valued case can be proved as in \cite[Section 4.6.4]{RS96}.

\begin{proposition} \label{prop:algebra1} Let $X$ and $Y$ be Banach spaces, $s>0$ $p\in(1,\infty)$, $q\in [1, \infty]$ and $w\in A_p$. Then for $\mathcal{A}\in \{B,F\}$ we have
\begin{align*}
\|m f\|_{\mathcal{A}^{s}_{p,q}(\R^d,w;X)} \leq C \big(\|m\|_{L^\infty(\R^d,w;\calL(X,Y))} \|f\|_{\mathcal{A}^{s}_{p,q}(\R^d,w;X)} + \|m\|_{\mathcal{A}^{s}_{p,q}(\R^d,w;\calL(X,Y))} \|f\|_{L^\infty(\R^d;X)}\big).
\end{align*}
\end{proposition}

In the vector-valued case one has $H^{s,p}(\R^d; X) = F_{p,2}^s(\R^d;X)$ if and only if $X$ can be renormed as a Hilbert space (see Remark \ref{HW} and Proposition \ref{prop:type-embed} below). Hence a vector-valued version of \eqref{eq:classicalcase2} is not contained in Proposition \ref{prop:algebra1}. To obtain a result in this direction for UMD-valued Bessel-potential spaces we make use of the notions type and cotype. These are measures for how far a space $X$ is away from being a Hilbert space.

Let a Rademacher sequence $(r_k)_{k\geq 0}$ on a probability space $\Omega$ be given, see Section \ref{subsec:UMD}. Then $X$ is said to have type $\tau\in [1,2]$ if there is $C > 0$ such that for all $N\in \N$ and $x_0,...,x_N\in X$ we have
$$\Big\| \sum_{n=0}^N r_n x_n \Big\|_{L^\tau(\Omega;X)} \leq C \Big( \sum_{n=0}^N  \|x_n\|_X^\tau \Big)^{1/\tau}.$$
Similarly, $X$ is said to have cotype $q\in [2,\infty]$ if
$$\Big( \sum_{n=0}^N  \|x_n\|_X^q \Big)^{1/q} \leq C \Big\| \sum_{n=0}^N r_n x_n \Big\|_{L^q(\Omega;X)}.$$
For a general overview on this topic we refer to \cite[Chapter 11]{DJT}. Some basic facts are as follows:
\begin{enumerate}[(a)]
\item Every Banach space has type $\tau = 1$ and cotype $q=\infty$.
\item If $X$ has type $\tau$, then it has type $\sigma$ for all $\sigma\in [1, \tau]$.
\item If $X$ has cotype $q$, then it has cotype $r$ for all $r\in [q, \infty]$.
\item A space $X$ can be renormed as a Hilbert space if and only if it has type $2$ and cotype $2$.
\item If $(S,\mu)$ is a $\sigma$-finite measure space, then $L^r(S)$ has type $\min\{2,r\}$ and cotype $\max\{2,r\}$.
\end{enumerate}
The connection of these notions to $X$-valued function spaces is as follows.
\begin{proposition} \label{prop:type-embed} Let $X$ have \emph{UMD}, $s\in \R$, $p\in (1,\infty)$ and $w\in A_p$. Assume $X$ has type $\tau\in [1,2]$ and cotype $q\in [2,\infty]$. Then
\begin{equation}\label{eq:typetauemb}
F^{s}_{p,\tau}(\R^d,w;X) \hookrightarrow H^{s,p}(\R^d,w;X) \hookrightarrow F^{s}_{p,q}(\R^d,w;X).
\end{equation}
\end{proposition}
\begin{proof} Using Proposition \ref{prop:UMDHisF}, this can be shown in the same way as in \cite[Proposition 3.1]{Ver12}.
\end{proof}

We have the following product estimate.

\begin{proposition} \label{multiplication-Hinfty}
Let $X$ and $Y$ have \emph{UMD}, $s > 0$, $p\in (1,\infty)$ and $w\in A_p$. Assume $Y$ has type $\tau\in (1,2]$ and that $m \in F^{s}_{p,\tau}(\R^d, w; \calL(X,Y))$ has $\mathcal R$-bounded image. Then
\begin{align*}
\|mf\|_{H^{s,p}(\R^d,w;Y)} \leq C \big(\mathcal R(m) \|f\|_{H^{s,p}(\R^d,w;X)} + \|m\|_{F^{s}_{p,\tau}(\R^d, w; \calL(X,Y))}  \|f\|_{L^\infty(\R^d;X)}\big).
\end{align*}
\end{proposition}
\begin{proof} We estimate the paraproducts $\Pi_{i}(m,f)$ for $i=1, 2, 3$. It follows from Lemma \ref{multiplication1} that
$$\|\Pi_1(m,f)\|_{H^{s,p}(\R^d,w;Y)} \leq C \mathcal R(m) \|f\|_{H^{s,p}(\R^d,w;X)}.$$
The summands of $\Pi_2(m,f)$ satisfy \eqref{FsupportPi2}. We use \eqref{eq:typetauemb}  and the estimate \eqref{6011} from Lemma \ref{para2} under the assumption \eqref{6001} to get
\begin{align*}
\|\Pi_2(m,f)\|_{H^{s,p}(\R^d,w;Y)} & \leq C\|\Pi_2(m,f)\|_{F^{s}_{p,\tau}(\R^d,w;Y)}
\\ & \leq C\sum_{j=-1}^1 \Big \| \Big ( 2^{sn} S_{n+j}m S_n f\Big)_{n \geq 0} \Big \|_{L^p(\R^d,w; \ell^\tau(Y))}
\\ & \leq C\sum_{j=-1}^1 \Big \| \Big ( 2^{sn} S_{n+j}m \Big)_{n \geq 0} \Big \|_{L^p(\R^d,w; \ell^\tau(\calL(X,Y)))} \sup_{n\geq 0} \|S_n f\|_{L^\infty(\R^d;X)}
\\ & \leq C  \|m\|_{F^{s}_{p,\tau}(\R^d, w; \calL(X,Y))} \|f\|_{L^\infty(\R^d;X)}.
\end{align*}
The estimate for $\Pi_3(m,f)$ is proved in the same way using \eqref{6011} under the assumption \eqref{6000}. \end{proof}

As a special case of this result we extend the classical estimate \eqref{eq:classicalcase2} to the weighted vector-valued setting with a scalar-valued multiplier. It in particular applies in case $X = L^r$ with $r \geq 2$, which is often the range of interest in the context of  nonlinear partial differential equations.

\begin{theorem}\label{thm:lastone}
Let $X$ be a \emph{UMD}-Banach space with type $\tau = 2$, let $s>0$, $p\in(1,\infty)$ and $w\in A_p$. Then \begin{align*}
\|mf\|_{H^{s,p}(\R^d,w;X)} \leq C \big(\|m\|_{L^\infty(\R^d)} \|f\|_{H^{s,p}(\R^d,w;X)} + \|m\|_{H^{s,p}(\R^d, w)} \|f\|_{L^\infty(\R^d;X)}\big).
\end{align*}
\end{theorem}
\begin{proof} This follows from Proposition \ref{multiplication-Hinfty} applied with $\tau = 2$, the fact that $\mathcal R(m) \leq 2 \|m\|_\infty$ for scalar-valued $m$ (see Remark \ref{rem:R-bound-scalar}) and that $F_{p,2}^s(\R^d,w) = H^{s,p}(\R^d,w)$ for an $A_p$-weight $w$ (see Proposition \ref{prop:type-embed}).
\end{proof}

\begin{appendix}

\section{Spaces of entire analytic functions}\label{sec:analytic}
In this appendix we consider weighted spaces with mixed norms of entire analytic functions, which are the key to convergence and estimates of the paraproducts. The results are weighted extensions of the corresponding assertions in \cite{Franke86}. Since some of the proofs differ from the unweighted case, we give all details in the proofs below.

For $A > 0$, $s\in \R$, $p\in (1,\infty)$, $q\in [1,\infty]$ and $w\in A_\infty(\R^d)$ we set
\begin{align*}
 L^p_A(\R^d,w;\ell^{s,q}(X))  = \big \{&\, (f_k)_{k\geq 0}\subset \TD(\R^d;X): \;\supp \wh{f}_k  \subset \{|\xi|\leq A 2^k\},\\
 &\; \|(f_k)_{k\geq 0}\|_{L^p_A(\R^d,w;\ell^{s,q}(X))} = \|(2^{sk} f_k)_{k\geq 0}\|_{L^p(\R^d,w; \ell^q(X))} < \infty \big \}, \\
 \ell^{s,q}(L^p_A(\R^d,w;X))  = \big \{&\, (f_k)_{k\geq 0}\subset \TD(\R^d;X):\; \supp \wh{f}_k  \subset \{|\xi|\leq A 2^k\},\\
 &\; \|(f_k)_{k\geq 0}\|_{\ell^{s,q}(L^p_A(\R^d,w;X))} =  \|(2^{sk} f_k)_{k\geq 0}\|_{\ell^q(L^p(\R^d,w;X))} < \infty \big \}.
\end{align*}
For $p,r\in (1,\infty)$ and $w\in A_{\infty}(\R)$ we further consider the mixed-norm space
$$L^{p(r)}(\R^d,w;X) = L^p(\R^{d-1}; L^r(\R, w;X)),$$
where the weight $w$ is understood to depend on the last coordinate only. The norm in this space given by
$$\|f\|_{L^{p(r)}(\R^d,w;X)}^p = \int_{\R^{d-1}} \|f(x',\cdot)\|_{L^r(\R,w;X)}^p \,dx'.$$
Then, with parameters as before,
\begin{align*}
 \ell^{s,q}(L_A^{p(r)}(\R^d,w;X)) = \big \{&\, (f_k)_{k\geq 0}\subset \TD(\R^d;X):\; \supp \wh f_k \subset \{|\xi|\leq A 2^k\},\\
 &\; \|(f_k)_{k\geq 0}\|_{\ell^{s,q}(L_A^{p(r)}(\R^d,w;X))} = \|(2^{sk} f_k)_{k\geq 0}\|_{\ell^q(L^{p(r)}(\R^d,w;X))} < \infty \big \}.
\end{align*}

\subsection{A maximal inequality} \label{sec:max-ineq}
Let $M$ be the Hardy-Littlewood maximal operator introduced in Section \ref{sec:Mucken}. The following extension of the Fefferman-Stein inequality \eqref{Fefferman-Stein} to spaces with mixed norms is straightforward to prove.

\begin{lemma}\label{lem:maxoperator}
Let $p,r \in (1, \infty)$, $q\in (1, \infty]$ and $w\in A_{p}(\R)$. Then
$$\|(M f_n)_{n\geq 0}\|_{ L^{p(r)}(\R^{d},w;\ell^q)} \leq C \|(f_n)_{n\geq 0}\|_{ L^{p(r)}(\R^{d},w;\ell^q)}.$$
\end{lemma}
\begin{proof}

\emph{Step 1}. Assume $1<q<\infty$. Let $M'$ denote the maximal operator with respect to $x'\in \R^{d-1}$ and $M''$ the maximal operator with respect to $t\in \R$.
It is elementary to check that there is $C > 0$ such that the pointwise estimate
$$M g \leq C M'' M' g, \qquad g\in L^1_{\text{loc}}(\R^{d}),$$
holds true. Now let $(f_n)_{n\geq 0}\in L^{p(r)}(\R^{d},w;\ell^q)$. For almost all fixed $x'$ we estimate, using the Fefferman-Stein inequality \eqref{Fefferman-Stein} on $L^r(\R,w;\ell^q)$ with respect to $M''$,
\begin{align*}
\big \|\big(M f_n(x',\cdot)\big)_{n\geq 0}\big\|_{L^{r}(\R,w;\ell^q)}&\, \leq C \big\|\big(M''M' f_n(x',\cdot)\big)_{n\geq 0}\big\|_{L^{r}(\R,w;\ell^q)}\\
&\,  \leq C \big\|\big(M' f_n(x',\cdot)\big)_{n\geq 0}\big\|_{L^{r}(\R,w;\ell^q)}.
\end{align*}
Applying the $L^{p}(\R^{d-1})$-norm, we find
\begin{equation}\label{bla}
\|(M f_n)_{n\geq 0}\|_{L^{p(r)}(\R^{d},w;\ell^q)}\leq C\|(M' f_n)_{n\geq 0}\|_{L^{p(r)}(\R^{d},w;\ell^q)}.
\end{equation}
Now let $Y = L^{r}(\R,w;\ell^q)$. Then $Y$ is a UMD Banach lattice, see \cite[Proposition 3]{RF}. We may therefore apply the maximal inequality from \cite[Theorem 3]{RF} on $L^{p}(\R^{d-1}; Y) = L^{p(r)}(\R^{d},w;\ell^q)$ to obtain
$$
\|(M' f_n)_{n\geq 0}\|_{L^{p}(\R^{d-1}; Y)} \leq C \|(f_n)_{n\geq 0}\|_{L^{p}(\R^{d-1}; Y)}.
$$
Combining this with \eqref{bla} gives the asserted maximal inequality.

\emph{Step 2}. For the case $q=\infty$ we note that
$$\|(M f_n)_{n\geq 0}\|_{L^{p(r)}(\R^{d},w;\ell^\infty)} \leq \|M g\|_{L^{p(r)}(\R^{d},w)},$$
where $g(x) = \sup_n|f_n(x)|$. Now one can argue as in Step 1 on $L^p(\R^{d-1};L^{r}(\R^{d_1},w))$.
\end{proof}

\subsection{Embeddings of Jawerth-Franke type} \label{sec:JF}
The following weighted Jawerth-Franke type embeddings for spaces over the real line are proved in \cite[Theorem 6.4]{MeyVer1}. The striking point is the independence of the microscopic parameter $q$.

\begin{proposition} \label{prop:JF} Let  $s_0> s_1$, $1<p_0<p_1<\infty$, $\gamma_0\in (-1,p_0-1)$ and $\gamma_1\in (-1,p_1-1)$. Assume
$$\frac{\gamma_0}{p_0} \geq \frac{\gamma_1}{p_1}, \qquad s_0 - \frac{1+\gamma_0}{p_0} \geq s_1 - \frac{1+\gamma_1}{p_1}.$$
Then for $q\in [1,\infty]$ one has the continuous embeddings
\begin{equation}\label{jf-BF}
B_{p_0,p_1}^{s_0}(\R,w_{\gamma_0};X) \hookrightarrow F_{p_1,q}^{s_1}(\R,w_{\gamma_1};X),
\end{equation}
\begin{equation}\label{jf-FB}
F_{p_0,q}^{s_0}(\R,w_{\gamma_0};X) \hookrightarrow B_{p_1,p_0}^{s_1}(\R,w_{\gamma_1};X).
\end{equation}
\end{proposition}

As in \cite{Franke86, RS96}, we need discrete versions of these embeddings on the spaces of entire analytic functions. We follow \cite[Section 2.3]{Franke86}, see also \cite[Section 2.6.3]{RS96}. As a preparation we state the following elementary result on Fourier supports.

\begin{lemma}\label{lem:Fourier-supp}
Let $f\in \TD(\R^d;X)$ be such that $\supp \wh f\subseteq \{|\xi|\leq A\}$ for some $A > 0$. Denote by $\F_t$ the Fourier transform with respect to the last coordinate $t\in \R$. Then for each $x'\in \R^{d-1}$ we have $\supp\F_t(f(x',\cdot))\subseteq \{|\lambda|\leq A\}$.
\end{lemma}

We have the following extension of Proposition \ref{prop:JF}. Observe that \eqref{jf_discrete_BF} and \eqref{jf_discrete_FB} correspond to \eqref{jf-BF} and \eqref{jf-FB}, respectively. The corresponding results in the unweighted case can be found in \cite[Theorem 2.4.1(IV)]{Franke86} and \cite[Theorem 2.6.3/3]{RS96}. We argue as in \cite{Franke86}, where the proof is only indicated.

\begin{proposition} \label{prop:jf_discrete} Let   $A > 0$, $s_0 > s_1$, $1<p_0<p_1< \infty$, $\gamma_0 \in (-1,p_0-1)$ and $\gamma_1\in (-1,p_1-1)$. Assume
$$\frac{\gamma_0}{p_0} \geq \frac{\gamma_1}{p_1}, \qquad s_0 - \frac{1+\gamma_0}{p_0} \geq s_1 - \frac{1+\gamma_1}{p_1}.$$
Then for $q\in [1,\infty]$ one has the continuous embeddings
\begin{equation}\label{jf_discrete_BF}
 \ell^{s_0,p_1}(L_A^{p_1(p_0)}(\R^d,w_{\gamma_0};X)) \hookrightarrow  L^{p_1}_A(\R^d,w_{\gamma_1};\ell^{s_1,q}(X)),
\end{equation}
 \begin{equation}\label{jf_discrete_FB}
  L^{p_0}_A(\R^d,w_{\gamma_0};\ell^{s_0,q}(X)) \hookrightarrow \ell^{s_1,p_0}(L_A^{p_0(p_1)}(\R^{d},w_{\gamma_1};X)).
 \end{equation}
\end{proposition}

\begin{proof} \emph{Step 1.} Take the smallest integer $N$ such that $A\leq 2^N$  and set $\text{e}_k(t) = e^{\text{i}2^{N+3+k}t}$. For $s\in \R$, $p\in (1,\infty)$, $q\in [1,\infty]$, $\gamma\in (-1,p-1)$ and  $(f_k)_{k\geq 0} \subset \TD(\R;X)$ with $\supp \wh f_k \subset\{|\xi|\leq 2^{N+k}\}$ we claim that
\begin{equation}\label{eq:equiv-F}
 \|(f_k)_{k\geq 0}\|_{L^{p}_A(\R,w_{\gamma};\ell^{s,q}(X))} \eqsim \Big \|\sum_{k\geq 0} \text{e}_k f_k\Big\|_{F_{p,q}^{s}(\R,w_{\gamma};X)},
\end{equation}
\begin{equation}\label{eq:equiv-B}
 \|(f_k)_{k\geq 0}\|_{\ell^{s,q}(L_A^{p}(\R,w_{\gamma};X))} \eqsim \Big \|\sum_{k\geq 0} \text{e}_k f_k\Big\|_{B_{p,q}^{s}(\R,w_{\gamma};X)}.
\end{equation}
Here $a\eqsim b$ means $C^{-1} a\leq b\leq Ca$. Let us prove \eqref{eq:equiv-F}, the case \eqref{eq:equiv-B} is similar. Observe  that $\supp \F(\text{e}_kf_k) \subseteq \{7\cdot 2^{N+k} \leq |\xi| \leq 9 \cdot 2^{N+k}\}$. Let $(S_n)_{n\geq 0}$ be defined with respect to $(\varphi_n)_{n\geq 0}\in \Phi(\R)$.
Then, for $n,k\geq 1$,
$$S_n (\text{e}_kf_k) \neq 0 \qquad \text{only if }\;\;   k+l_0 \leq n \leq k+l_1,$$
where $l_0,l_1\in \N$ are independent of $n$ and $k$.
We use this, \cite[Proposition 2.4]{MeyVer1} and that $\wh{\varphi}_n = \wh{\varphi}_1 (2^{-n+1\cdot})$ to obtain (setting $f_k = 0$ for negative $k$)
\begin{align*}
\Big \|\sum_{k\geq 0} &\,\text{e}_k f_k\Big\|_{F_{p,q}^{s}(\R,w_{\gamma};X)} \leq \sum_{l= l_0}^{l_1} \big \| \big(  S_{n} (\text{e}_{n-l} f_{n-l})\big)_{n\geq 0}\big\|_{L^{p}_A(\R,w_{\gamma};\ell^{s,q}(X))}\\
&\, \leq C\sum_{l= l_0}^{l_1} \sup_{n\geq 0} \| (1+|\cdot|^{2}) \F^{-1}\big( \wh{\varphi}_n (2^{N+n+l+1}\cdot)\big)\|_{L^1(\R)} \big \| ( \text{e}_{n-l} f_{n-l})_{n\geq 0}\big\|_{L^{p}_A(\R,w_{\gamma};\ell^{s,q}(X))}\\
&\, \leq C  \big \|  ( f_{n})_{n\geq 0}\big\|_{L^{p}_A(\R,w_{\gamma};\ell^{s,q}(X))}.
\end{align*}
For the converse we note that the Fourier supports of the $\text{e}_kf_k$ are pairwise disjoint. Take a function $\psi\in \Schw(\R)$ such that  $\wh{\psi} \equiv 1$ on $ \{7\cdot 2^N \leq |\xi| \leq 9 \cdot 2^N \}$ and $\wh{\psi} \equiv 0$ on $\{|\xi|\leq 6 \cdot 2^N\}\cup \{|\xi|\geq 10 \cdot 2^N\}$. Define $\psi_k$ by $\wh{\psi}_k = \wh{\psi}(2^{-k}\cdot)$.
Using again \cite[Proposition 2.4]{MeyVer1}, we get
\begin{align*}
\|(f_k)_{k\geq 0} &\,\|_{L^{p}_A(\R,w_{\gamma};\ell^{s,q}(X))}\leq  \sum_{l=l_0}^{l_1}  \| \big (S_{k+l} (\text{e}_k f_k)\big)_{k\geq 0}\|_{L^{p}_A(\R,w_{\gamma};\ell^{s,q}(X))}\\
&\, =  \sum_{l=l_0}^{l_1}\Big \| \Big( \psi_k* S_{k+l} \sum_{j\geq 0} \text{e}_j f_j\Big)_{k\geq 0}\Big\|_{L^{p}_A(\R,w_{\gamma};\ell^{s,q}(X))}\\
 &\,  \leq C\sum_{l=l_0}^{l_1} \sup_{k\geq 0} \big \| (1+|\cdot|^{2}) \F^{-1}\big( \wh{\psi}_k (2^{N+k+1}\cdot)\big)\big \|_{L^1(\R)} \Big \| \Big( S_{k+l}\sum_{j\geq 0} \,\text{e}_j f_j\Big)_{k\geq 0}\Big\|_{L^{p}_A(\R,w_{\gamma};\ell^{s,q}(X))}\\
 &\, \leq C \Big \|\sum_{j\geq 0} \text{e}_j f_j\Big\|_{F_{p,q}^{s}(\R,w_{\gamma};X)}.
\end{align*}

\emph{Step 2.} To prove \eqref{jf_discrete_BF}, let $(f_k)_{k\geq 0} \in \ell^{s_0,p_1}(L_A^{p_1(p_0)}(\R^d,w_{\gamma_0};X))$. Then  $\supp \F_t (f_k(x',\cdot)) \subseteq \{|\lambda|\leq A 2^k\}$ for $x'\in \R^{d-1}$ and each $k$ by Lemma \ref{lem:Fourier-supp}, where $\F_t$ is the Fourier transform with respect to $t\in \R$. We may thus use the equivalences \eqref{eq:equiv-F} and \eqref{eq:equiv-B} together with \eqref{jf-BF} to estimate
\begin{align*}
\|(f_k)_{k\geq 0}\|_{L^{p_1}_A(\R^d,w_{\gamma_1};\ell^{s_1,q}(X))}^{p_1} &\, = \int_{\R^{d-1}}\| (f_k(x',\cdot))_{k\geq 0}\|_{L^{p_1}_A(\R,w_{\gamma_1};\ell^{s_1,q}(X))}^{p_1}\, dx'\\
&\, \leq C \int_{\R^{d-1}} \Big\|\sum_{k\geq 0} \text{e}_k f_k(x',\cdot)\Big\|_{F_{p_1,q}^{s_1}(\R,w_{\gamma_1};X)}^{p_1} \, dx'\\
&\, \leq C \int_{\R^{d-1}} \Big\|\sum_{k\geq 0} \text{e}_k f_k(x',\cdot)\Big\|_{B_{p_0,p_1}^{s_0}(\R,w_{\gamma_0};X)}^{p_1} \, dx'\\
&\, \leq C \,\int_{\R^{d-1}} \| (\| f_k(x',\cdot)\|_{L^{p_0}(\R,w_{\gamma_0};X))})_{k\geq 0} \|_{\ell^{s_0,p_1}}^{p_1} \, dx'\\
&\, = C \|(f_k)_{k\geq 0}\|_{\ell^{s_0,p_1}(L_A^{p_1(p_0)}(\R^d,w_{\gamma_0};X))}^{p_1}.
\end{align*}
The derivation of \eqref{jf_discrete_FB} uses \eqref{jf-FB} and is analogous.
\end{proof}

\subsection{Convergence criteria for series}\label{subsec:criteria}
The following result provides sufficient conditions for the convergence of series in weighted mixed-norm spaces of entire analytic functions.  We refer to \cite[Section 2.3.2]{RS96} and \cite[Section 3.6]{JS08} for the unweighted cases.

\begin{lemma} \label{para2} Let $p,p_0,p_1\in (1,\infty)$, $q\in (1,\infty]$,  $w\in A_p(\R^d)$ and $w_1\in A_{p_1}(\R)$, where $w_1$ is understood to depend on the last coordinate $t\in \R$. Suppose that for some $k_0\in \N$ the sequence $(f_k)_{k\geq 0}\subset \TD(\R^d;X)$ and $s\in \R$ satisfy
\begin{equation}\label{6000} \text{either}\quad s\in \R\quad\text{and}\quad \supp \wh{f_0} \subset \{|\xi|\leq 2^{k_0}\},\quad  \supp \wh{f_k} \subset \{2^{k-k_0} \leq |\xi|\leq 2^{k+k_0}\}\,;
 \end{equation}
\begin{equation}\label{6001} \text{or}\quad s>0  \quad\text{and}\quad\supp \wh{f_k} \subset \{|\xi|\leq 2^{k+k_0}\}\;.
 \end{equation}
Then the following holds true.  If $(2^{sk} f_k)_{k\geq 0} \in L^{p_0(p_1)}(\R^d,w_1; \ell^q(X))$, then  $f = \sum_{k=0}^\infty f_k$ converges  in  $\TD(\R^d;X)$ and
\begin{equation}\label{6005}
\|(2^{sn}S_nf)_{n\geq 0}\|_{L^{p_0(p_1)}(\R^d,w_1; \ell^q(X))} \leq C \|(2^{sk} f_k)_{k\geq 0}\|_{L^{p_0(p_1)}(\R^d,w_1; \ell^q(X))}.
\end{equation}
In the same sense we have the estimates
\begin{equation}\label{6010}
\|(2^{sn}S_nf)_{n\geq 0}\|_{\ell^q(L^{p_0(p_1)}(\R^d,w_1; X))} \leq C \|(2^{sk} f_k)_{k\geq 0}\|_{\ell^q(L^{p_0(p_1)}(\R^d, w_1; X))},
\end{equation}
\begin{equation}\label{6011}
\|f\|_{F_{p,q}^s(\R^d,w;X)} \leq C \|(2^{sk} f_k)_{k\geq 0}\|_{L^p(\R^d, w; \ell^q(X))},
\end{equation}
\begin{equation}\label{6012}
\|f\|_{B_{p,q}^s(\R^d,w;X)} \leq C \|(2^{sk} f_k)_{k\geq 0}\|_{\ell^q(L^p(\R^d, w;X))}.
\end{equation}
Assuming \eqref{6000}, all assertions hold true also for $q = 1$ and $A_\infty$-weights. Assuming \eqref{6001}, the estimates \eqref{6010} and \eqref{6012} hold true also for $q = 1$.
\end{lemma}
\begin{proof} \emph{Step 1.} First assume $q \in (1,\infty]$ and that the weights are in $A_{p}$ and $A_{p_1}$, respectively. Throughout we set $f_k = 0$ for $k < 0$. Suppose that \eqref{6001} is satisfied. We show the convergence of the series and the estimate \eqref{6005}.

Fix $N\in \N$. Then for each $n$ the support condition for the $\wh f_k$ implies
$$S_n \sum_{k=0}^N f_k = S_n \sum_{k=n-k_0}^N f_k\quad \text{if }n \leq N+k_0, \qquad S_n \sum_{k=0}^N f_k = 0\quad \text{if }n > N+k_0.$$
Therefore, since $s> 0$,
\begin{align}
\Big\|\Big (2^{sn}S_n\sum_{k=0}^N f_k\Big)_{n\geq 0}\Big\|_{L^{p_0(p_1)}(\R^d,w_1; \ell^q(X))} &\,=  \Big \| \Big(2^{sn}S_n \sum_{l=-k_0}^{N-n} f_{n+l}  \Big)_{n\geq 0} \Big \|_{L^{p_0(p_1)}(\R^d,w_1;\ell^q(X))}\nonumber\\
&\,\leq \sum_{l=-k_0}^\infty 2^{-sl}\Big \| \Big(2^{s(n+l)}S_n  f_{n+l}  \Big)_{n\geq 0} \Big \|_{L^{p_0(p_1)}(\R^d,w_1;\ell^q(X))}\label{6003}\\
&\, \leq C\,\sup_{l\geq -k_0}\Big \| \Big(2^{s(n+l)}S_n  f_{n+l}  \Big)_{n\geq 0} \Big \|_{L^{p_0(p_1)}(\R^d,w_1;\ell^q(X))}, \label{6999}
\end{align}
where we set $\sum_{l=-k_0}^{N-n}$ equal to zero whenever $N-n<-k_0$.

To estimate the right-hand side of \eqref{6999}, define $\psi_n$ by $\psi_n(x) = \sup_{|y|\geq |x|} |\varphi_n(y)|$ and set $g_{n+l} = 2^{s(n+l)}\|f_{n+l}\|$. Applying \cite[Theorem 2.1.10]{GraClass} we find that for every $n\geq 0$,
\begin{align*}
\|2^{s(n+l)} S_n  f_{n+l}(x)\| \leq \|\psi_n\|_{L^1(\R^d)} Mg_{n+l}(x) \leq C Mg_{n+l}(x),
\end{align*}
where $M$ is the Hardy-Littlewood maximal operator. Lemma \ref{lem:maxoperator} gives
\begin{align*}
\big\| \big(2^{s(n+l)}S_n  f_{n+l}  \big)_{n\geq 0} \big\|_{L^{p_0(p_1)}(\R^d,w_1;\ell^q(X))} & \leq C \big\| (M g_{n+l} )_{n\geq 0} \big \|_{L^{p_0(p_1)}(\R^d,w_1;\ell^q)}
\\ & \leq C \big\| (g_{n+l})_{n\geq 0} \big\|_{L^{p_0(p_1)}(\R^d,w_1;\ell^q)}
\\ & = C\big\| \big(2^{s(n+l)}f_{n+l}  \big)_{n\geq 0} \big \|_{L^{p_0(p_1)}(\R^d,w_1;\ell^q(X))}.
\end{align*}
Combining this estimate with \eqref{6999}, we obtain
\begin{equation}\label{abc1}
 \Big\|\Big (2^{sn}S_n\sum_{k=0}^N f_k\Big)_{n\geq 0}\Big\|_{L^{p_0(p_1)}(\R^d,w_1; \ell^q(X))}  \leq C \big\| (2^{sk} f_{k})_{k\geq 0} \big\|_{L^{p_0(p_1)}(\R^d,w_1;\ell^q(X))},
\end{equation}
with a constant $C$ independent of $N$. Now set
\[\|g\|_{F_{p_0(p_1),q}^s(\R^d,w_1;X)} =  \big\|\big (2^{sn}S_n g\big)_{n\geq 0}\big\|_{L^{p_0(p_1)}(\R^d,w_1; \ell^q(X))}.\]
This defines a complete space of distributions, which embeds continuously into $\TD(\R^d;X)$ (see the proofs of \cite[Theorem 2.3.3]{Tri83} and \cite[Proposition 10]{JS07}, and use \cite[Lemma 4.5]{MeyVer1}). It follows from \eqref{abc1} that $(\sum_{k=0}^N f_k)_{N\geq 0}$ is a Cauchy sequence in $F_{p_0(p_1),1}^{s-\varepsilon}(\R^d,w_1;X)$ for $\varepsilon> 0$. Hence it converges in $\TD(\R^d;X)$. A Fatou argument as in \eqref{fatou} applied to \eqref{abc1} yields the estimate \eqref{6005}.

The other estimates can be derived in a similar way. In case when the Fourier supports satisfy \eqref{6000}, the sum $\sum_{l=-k_0}^\infty$ in \eqref{6003} can be replaced by $\sum_{l=-k_0}^{k_0}$. Then the restriction on $s$ is not necessary.

\emph{Step 2.} Consider the case $q = 1$. Assume \eqref{6001}. Then \eqref{6010} and \eqref{6012} can be shown as before, where instead of Lemma \ref{lem:maxoperator} it suffices to use the boundedness of $M$ on $L^{p_1}(\R,w_1)$ and on $L^p(\R^d,w)$, respectively.

Assume \eqref{6000} and $w_1\in A_\infty$. We prove \eqref{6005}, the arguments for the other estimates are similar. Arguing as before, we get
$$
\Big\|\Big (2^{sn}S_n\sum_{k=0}^N f_k\Big)_{n\geq 0}\Big\|_{L^{p_0(p_1)}(\R^d,w_1; \ell^1(X))}  \leq C\,\sum_{|l|\leq k_0}\Big \| \Big(2^{s(n+l)}S_n  f_{n+l}  \Big)_{n\geq 0} \Big \|_{L^{p_0(p_1)}(\R^d,w_1;\ell^1(X))}.
$$
Choose $r\in (0,1)$ such that $w_1\in A_{p_1/r}(\R)$. For $x\in \R^d$ we have
\begin{align*}
\|S_n f_{n+l}(x)\| \leq \sup_{z\in \R^d} \frac{\|f_{n+l}(x-z)\|}{1+ |2^nz|^{d/r}} \int_{\R^d} (1+ |2^ny|^{d/r}) |\varphi_n(y)|\,d y.
\end{align*}
Here the second factor is bounded independent of $n$ since $\varphi_n = 2^{nd} \varphi_1(2^{n-1}\cdot)$.
The diameter of the Fourier support of $f_{n+l}$ is comparable to $2^n$. We thus obtain from the proof of \cite[Theorem 1.6.2]{Tri83} that
\begin{align*}
2^{s(n+l)} \|S_n f_{n+l}(x)\| &\, \leq C 2^{s(n+l)} \sup_{z\in \R^d} \frac{\|f_{n+l}(x-z)\|}{1+ |2^nz|^{d/r}} \\
&\, \leq C 2^{s(n+l)}(M\|f_{n+l}\|^r(x))^{1/r} = C (Mg_{n+l}^r(x))^{1/r},
\end{align*}
where as above $g_{n+l}  = 2^{s(n+l)}\|f_{n+l}\|$. Since  $1/r> 1$ and $w_1\in A_{p_1/r}(\R)$, we can use Lemma \ref{lem:maxoperator} to estimate
\begin{align*}
\Big \| \Big(2^{s(n+l)}S_n f_{n+l}  \Big)_{n\geq 0} \Big \|_{L^{p_0(p_1)}(\R^d,w_1;\ell^1(X))} &\leq
C \big\|(Mg_{n+l}^r)_{n\geq 0}\big\|_{L^{p_0/r(p_1/r)}(\R^d,w_1;\ell^{1/r})}^{1/r}
\\ & \leq C\big\|(g_{n+l}^r)_{n\geq 0}\big\|_{L^{p_0/r(p_1/r)}(\R^d,w_1;\ell^{1/r})}^{1/r}
\\ & = C\big \|(2^{s(n+l)} f_{n+l})_{n\geq 0}\big\|_{L^{p_0(p_1)}(\R,w_1;\ell^{1}(X))}.
\end{align*}
Now the proof can be finished as before.
\end{proof}

\begin{remark}
We do not know how to prove \eqref{6005} and \eqref{6011} under the assumption \eqref{6001} for  $q = 1$ and $A_{\infty}$-weights. The above argument does not work since the supports of the $\wh f_{n}$ are too large.
\end{remark}

\end{appendix}

\end{document}